\documentclass[11pt,leqno]{amsart}
\usepackage{amsfonts,amssymb}

\newtheorem{theorem}{Theorem}[section]
\newtheorem{lemma}[theorem]{Lemma}
\newtheorem{proposition}[theorem]{Proposition}
\newtheorem{corollary}[theorem]{Corollary}

\newtheorem{remark}[theorem]{Remark}

\newcommand{\R}{{\mathbb R}}
\newcommand{\Z}{{\mathbb Z}}
\newcommand{\N}{{\mathbb N}}
\newcommand{\C}{{\mathbb C}}
\newcommand{\E}{{\mathbb E}}

\newcommand{\RE}{{\rm Re\ }}

\newcommand{\dd}{{\rm\; d}}
\newcommand{\ddc}{{\rm d}}

\newcommand{\trace}{{\mathrm {Tr \,}}}
\newcommand{\KER}{{\rm Ker}}

\newcommand{\cL}{{\mathcal L}}

\newcommand{\LL}{\mathcal{L}}

\newcommand{\NN}{\mathcal{N}}

\newcommand{\cP}{\mathcal{P}}

\newcommand{\cK}{\mathcal{L}}

\renewcommand {\Re}{{\rm Re}}

\newcommand{\D}{{\mathrm D}}

\numberwithin{equation}{section}
\begin{document}

\title[Nonautonomous Ornstein-Uhlenbeck Equations]{Asymptotic behavior
and hypercontractivity 
in nonautonomous Ornstein-Uhlenbeck equations}

\author{Matthias Geissert, Alessandra Lunardi}

\address{FB Mathematik\\ Schlossgartenstr. 7\\ TU Darmstadt\\ 64289
Darmstadt\\ Germany}
\email{geissert@mathematik.tu-darmstadt.de}

\address{Dipartimento di Matematica\\ Parco Area delle Scienze 53/A \\ 
 43100 Parma\\ Italia}
\email{alessandra.lunardi@unipr.it}

\subjclass[2000]{47D06, 47F05, 35B65}
\keywords{Nonautonomous PDEs, Ornstein-Uhlenbeck operators, asymptotic 
behavior, spectrum, hypercontractivity}

\thanks{Both authors were partially supported by the MIUR - PRIN 2004 Research 
project ``Kolmogorov Equations''.}

\begin{abstract}
In this paper we investigate a class of nonautonomous linear parabolic
problems with time-depending Ornstein-Uhlenbeck operators. We study the
asymptotic behavior of the associated evolution operator and  evolution semigroup in the periodic and non-periodic
situation. Moreover, we show that the
associated evolution operator is hypercontractive. 
\end{abstract}

\maketitle

\section{Introduction}

In this paper we continue the investigations of \cite{DPL06,GL07} on a 
class of nonautonomous linear parabolic problems with time-depending 
Ornstein-Uhlenbeck operators. We study asymptotic behavior and hypercontractivity in 
Cauchy problems, 
\begin{equation}
\label{e6}
\left\{\begin{array}{l}
u_s(s,x) + \cK(s)u(s,x) = 0, \;\;s\leq t, \;x\in \R^{n},
\\
\\
u(t)=\varphi(x), \;x\in \R^{n}, 
\end{array}\right.
\end{equation}
as well as equations with time in the whole $\R$ and no initial or 
final data, 
\begin{equation}
\label{e7}
\lambda u(s,x) - (u_s(s,x) + \cK(s)u(s,x)) = h(s,x), \;\;s\in \R, \;x\in \R^{n}.
\end{equation}
Here $(\cK(t))_{t\in \R}$ is a family of  Ornstein-Uhlenbeck operators, 
\begin{equation}
    \cK(t)\varphi(x)=\frac12 \trace \left( B(t)B^*(t)\D_x^2\varphi
    (x)\right)+ \langle A(t)x+f(t),\D_x\varphi(x)\rangle ,\quad x\in\R^n,
    \label{ouoperator}
\end{equation}
with continuous and bounded data $A,B:\R\to\cL(\R^n)$ and $f:\R\to\R^n$.
Throughout the paper  we assume that the operators $\cK$ are uniformly elliptic, 
i.e.  there exists $\mu_0>0$ such that
\begin{equation}
    \|B(t)x\|\geq\mu_0\|x\|,\quad t\in\R,\ x\in\R^n.
    \label{invertb}
\end{equation}

The backward Cauchy problem \eqref{e6} 
is the Kolmogorov equation of the nonautonomous stochastic ODE 
\begin{equation}
\label{e1}
\left\{\begin{array}{l}
\ddc X_{t}=(A(t)X_{t}+f(t))\ddc t+B(t)\ddc W(t), 
\\
\\
X_{s}=x,
\end{array}\right.
\end{equation}
where  $W(t)$ is a standard $n$-dimensional Brownian motion and $s\in 
\R$, $x\in \R^n$. Indeed,  denoting by 
$X(s,t,x)$ the solution to \eqref{e1}, for each $t\in \R$ and 
$\varphi \in 
C^{2}_{b}(\R^{n})$ the function $u(s,x) :={\mathbb 
E}(\varphi(X(s,t,x)))$ satisfies \eqref{e6}. See e.g. \cite{GS72,KS91}. 
 
Under our ellipticity assumption, 
$u$ is in fact a classical solution to  \eqref{e6} just for 
$\varphi\in C_b(\R^n)$. The
transition evolution operator $P_{s,t}\varphi(x) := \E[\varphi(X(t,s,x))]$
may be explicitly written as
\begin{equation}
    \label{Pst}    P_{s,t}\varphi(x) = 
\int_{\R^n}\varphi(y)\NN _{m(t,s),Q(t,s)}(dy)
,\quad \varphi\in C_b(\R^n),\; s\leq t.  
\end{equation}
Here $\NN _{m(t,s),Q(t,s)}$ is the Gaussian measure with mean $m(t,s)$ 
and covariance $Q(t,s)$ given respectively by 
\begin{equation}
\label{e2.4}
m(t,s) := U(t,s)x+\int_{s}^{t}U(t,r)f(r)dr , \quad Q(t,s):=\int_s^t 
U(t,r)B(r)B^*(r)U^*(t,r)dr ,
\end{equation}
and $U$ is the evolution operator for $A(\cdot)$, i.e. for each $x\in 
\R^{n}$ the function $t\mapsto U(t,s)x$ is the solution to $\xi '(t) = 
A(t)\xi(t)$, $\xi(s) = x$.  

In the autonomous elliptic case $B(t)\equiv B$, $A(t)\equiv A$, $f(t)\equiv 
0$, with   $\det B \neq 0$, 
we have $P_{s,t} = T(t-s)$ where $T(t)$ is the Ornstein-Uhlenbeck semigroup. 
$T(t)$ is a Markov semigroup in $C_{b}(\R^{n})$. Its asymptotic 
behavior is well understood in the case that all the 
eigenvalues of $A$ have negative real part, so that 
$\|e^{tA}\|$ decays exponentially as $t\to \infty$. In this case, for 
each $x\in \R^{n}$
$T(t)\varphi(x)$ converges to a constant which is the mean value of 
$\varphi$ with respect to the unique invariant measure $\mu = \NN 
_{0,Q_{\infty}}$ of $T(t)$, i.e. the unique Borel probability measure 
in $\R^{n}$ such that 
$$\int_{\R^{n}} T(t)\varphi \dd \mu = \int_{\R^{n}} \varphi \dd \mu , 
\quad t>0, \;\varphi \in   C_{b}(\R^{n}). $$
For each $p\in [1, +\infty)$, $T(t)$ is extended in a standard way to a contraction semigroup (still 
denoted by $T(t)$) in $L^{p}(\R^{n}, \mu)$. 
If $\varphi\in L^{p}(\R^{n}, \mu)$, 
then $ T(t)\varphi$ converges exponentially to 
the mean value of $\varphi$ in $ L^{p}(\R^{n}, \mu)$, and the rate of 
convergence coincides with the rate of decay of $\|e^{tA}\|$ to zero. Moreover, $T(t)$ is hypercontractive, i.e. for $p>1$ and $t>0$ it maps $L^{p}(\R^{n}, \mu)$ into $L^{q(t)}(\R^{n}, \mu)$ for a suitable $q(t)>p$, and with norm $\leq 1$.

In our nonautonomous case the assumption that $\|e^{tA}\|$ decays exponentially 
as $t\to \infty$ is replaced by the assumption that   $\|U(t,s)\|$  decays exponentially 
as $t-s \to \infty$. More precisely we assume that
\begin{equation}
\label{exp_decay}
\begin{array}{ll}
\omega_0(U) := \inf\{ & \omega \in \R :\, \exists M=M(\omega) \,{\rm  
such}\,{\rm that}
\\
\\
 & \|U(t,s)\|\leq M{\rm e}^{ \omega(t-s)},\quad -\infty<s\leq 
t<\infty \} <0.
\end{array}
\end{equation}
Then there is not a unique invariant measure, but there exist  families of 
Borel probability measures $\{\nu_{t}:\;t\in \R\}$, called  {\em entrance laws 
at time} $-\infty$ in \cite{Dyn89} and 
{\em  evolution systems of measures} in \cite{DPR05},  such 
that  
\begin{equation}
\label{evmeas}
\int_{ \R^n} P_{s,t}\varphi  \dd   \nu_s = 
\int_{ \R^n} \varphi  \dd   \nu_t, \quad \varphi \in 
C_{b}(\R^{n}), \;s\leq t.
\end{equation}
Such families are infinitely many, and they were characterized in 
\cite{GL07}. Among all of them, a distinguished one has a prominent 
role in  
the asymptotic behavior of $P_{s,t}$. It is the family of measures $\nu_t$
defined by
\begin{equation}
\label{nu_t}
 \nu_t=\NN_{g(t,-\infty),Q(t,-\infty)},\quad t\in\R,
\end{equation}
and it is the unique one with uniformly bounded moments of some order, i.e. 
there exists $\alpha >0$ such that 
\begin{equation}
\label{moments}
\sup_{t\in \R} \int\limits_{\R^n} 
|x|^{\alpha}\nu_t(\ddc x) <+\infty.
\end{equation}
In fact, it satisfies \eqref{moments} for each $\alpha >0$.
This implies that for each $\varphi \in C_{b}(\R^{n})$ and for each 
$t\in \R$, $x\in \R^{n}$ we have
$$\lim_{ s\to -\infty}P_{s,t}\varphi(x) = \int_{\R^{n}} 
\varphi(y)\ddc \nu_{t}.$$
As in the autonomous case, we have a much better behavior if we work 
in $L^{p}$ spaces with respect to the measures $\nu_{t}$. But in this 
context, the 
evolution operator $P_{s,t}$ maps $L^{p}(\R^{n}, \nu_{t})$ into 
$L^{p}(\R^{n}, \nu_{s})$, hence it cannot be seen as an evolution 
operator in a fixed Banach space $X$. Still, we have the contraction 
estimate
$$\| P_{s,t}\|_{\cL(L^p(\R^{n},\nu_t),L^p(\R^{n},\nu_s))}\leq 1, 
\quad s<t, $$
as well as smoothing  estimates, proved in 
\cite{GL07}, that are optimal both 
for $t-s$ close to $0$ and for $t-s\to \infty$, and that are quite 
similar to the corresponding estimates in the autonomous case:
\begin{equation}
\label{stime}
\|\D_x^\alpha P_{s,t}\|_{\cL(L^p(\R^{n},\nu_t),L^p(\R^{n},\nu_s))}\leq
\left\{
\begin{array}{rl}
  C(t-s)^{-|\alpha|/2}{\rm e}^{ \omega|\alpha|(t-s)},&0<t-s<1,
    \\
  C{\rm e}^{ \omega|\alpha|(t-s)},&t-s>1.
\end{array}
\right.
\end{equation}
Here $\alpha$ is any multi-index, $\omega$ is any number in $(\omega_{0}(U),0)$ 
and $C=C(\alpha, \omega)$. 

Such estimates are the starting point for our study of asymptotic 
behavior in the $L^2$ setting. 
As in the theory of ordinary differential equations, we get very precise asymptotic 
behavior results if the data are time periodic. In this case the asymptotic behavior of the 
evolution operator $P_{s,t}$ 
is driven by the spectral properties of $P_{0,T}$, where $T$ is the 
period. Note that $P_{0,T}$ is a bounded operator in $L^{2}(\R^{n}, 
\nu_{0}) $ since $\nu_{0}= \nu_{T}$. By estimates \eqref{stime}, 
$P_{0,T}$ is bounded from $L^{2}(\R^{n}, 
\nu_{0}) $ to $H^{1}(\R^{n}, 
\nu_{0}) $, which is compactly embedded in $L^{2}(\R^{n}, 
\nu_{0}) $ since $\nu_{0}$ is a Gaussian measure with nondegenerate covariance matrix. Therefore, its 
spectrum consists of $0$, plus (at most) a sequence of eigenvalues. 
We show that the 
unique eigenvalue of $P_{0,T}$ in the unit circle is $1$, that it has eigenvalues with 
modulus equal to $\exp( \omega_{0}(U)T)$, and that the modulus of the other eigenvalues 
 does not exceed $\exp( \omega_{0}(U)T)$. 

For any $t\in \R$ and $\varphi \in L^{2}(\R^{n}, \nu_{t}) $ let
$$M_{t}\varphi := \int_{\R^{n}} \varphi \dd \nu_{t}$$ 
be the mean value of $\varphi $ with respect to $\nu_{t}$. 
We know from \cite{DPL06} that  the $L^{2}(\R^{n}, \nu_{s})$-norm of
$P_{s,t}(\varphi - M_{t}\varphi)$ 
converges exponentially to $0$ as $t-s\to \infty$. Using the above 
spectral properties, we determine the exact convergence rate, proving that 
for each $\omega \in (\omega_{0}(U),0)$ there is  $M>0$ such 
that 
\begin{equation}
    \label{asintoticoI}
\|P_{s,t}(\varphi - M_{t}\varphi)\|_{L^{2}(\R^{n}, \nu_{s})} 
    \leq Me^{\omega(t-s)}\|\varphi\|_{L^{2}(\R^{n}, \nu_{t})} , \quad 
    s<t, \; \varphi\in L^2(\R^n,\nu_{t }), 
\end{equation}
and that for each $\omega < \omega_{0}(U)$ there is no $M $ such 
that   \eqref{asintoticoI} holds. Moreover, \eqref{asintoticoI} holds 
also for $\omega = \omega_{0}(U)$ iff all the eigenvalues of $U(T,0)$ 
with modulus equal to $\exp(T\omega_{0}(U))$ are semisimple.

Still in the case of $T$-periodic coefficients, a natural setting for 
problem \eqref{e7} is the space $L^2_\#( \R^{1+n},\nu)$ consisting of 
the Lebesgue measurable functions $h$ such that $h(s+T,x)=h(s,x)$ a.e. 
and the norm
$$\|h\|_{L^2_\#( \R^{1+n},\nu)} = \bigg( \frac{1}{T}\int_{0}^{T} \int_{\R^{n}} 
|h(s,x)|^{2}\dd \nu_{s} \ddc s\bigg)^{1/2}$$
is finite. In the paper \cite{GL07} we showed that if $\lambda  $ is 
any complex number, $h\in L^2_\#( \R^{1+n},\nu)$,
and $u\in L^2_\#( \R^{1+n},\nu)$ $\cap$ $H^{1,2}_{loc}(\R^{1+n}, \ddc t\times \ddc x)$ is a time 
periodic solution of \eqref{e7}, then $u$ belongs to $H^{1,2}_\#( 
\R^{1+n},\nu)$ i.e. $u_{t}$ and all the space derivatives $u_{x_{i}x_{j}}$ 
belong to $L^2_\#( \R^{1+n},\nu)$. The operator
$$\left\{ \begin{array}{l}
G_\#: D(G_\#)=  H^{1,2}_\#(\R^{1+n},\nu)\mapsto L^2_\#( \R^{1+n},\nu), 
\\
\\
G_\#u(s,x) =  u_s(s,x) + \cK(s)u(s,x) 
\end{array}\right. $$
may be seen as the infinitesimal generator of the evolution semigroup $ 
\cP^\#_\tau u $ in $L^2_\#( \R^{1+n},\nu)$ defined by
\begin{align}\label{rep}
  ( \cP^\#_\tau u  )(s,x) & = \left( P_{s,s+\tau}
    u(s+\tau,\cdot) \right)(x),\quad s\in\R,\ x\in\R^n,\
    \tau\geq 0 ,\ u\in L^2_\#( \R^{1+n},\nu),
\end{align}
and the measure $\nu$ is invariant for the semigroup $(\cP^\#_\tau)_{\tau \geq 0}$, see \cite{DPL06}. 
Although $(\cP^\#_\tau)_{\tau \geq 0}$ is not a standard evolution 
semigroup (since, as we already remarked, $P_{s,s+\tau}$ does not act 
in a fixed Banach space $X$ but it maps $L^{2}(\R^{n}, \nu_{s+\tau})$ into 
$L^{2}(\R^{n}, \nu_{s})$), a part of the classical theory of evolution 
semigroups may be extended to our situation, and the
spectral properties of the generator $G_\#$ are strongly connected 
with the asymptotic behavior of $ \cP^\#_\tau $. In its turn, the 
asymptotic behavior of $ \cP^\#_\tau $ may be easily deduced from the 
asymptotic behavior of $P_{s,t}$. 
In particular, setting 
\begin{equation}
    \label{pi}
(\Pi u )(t,x) := M_{t}u(t, \cdot), \quad  t\in \R, \;x\in \R^{n}, 
\end{equation}
and using \eqref{asintoticoI}, we see that $\cP^\#_\tau u$ converges exponentially to $\Pi u$ as 
$\tau \to \infty$, for each $u\in L^2_\#( \R^{1+n},\nu)$, and 
the growth bound of $(\cP^\#_\tau(I - \Pi))_{\tau\geq 0}$ is 
$\omega_{0}(U)$. $\Pi$ is the spectral projection relative to 
$\sigma(G_\#)\cap i\R =  2\pi i\Z /T $, its range is
isomorphic to $L^2_\#(\R; \ddc t)$. Moreover, $G_\#$ has infinitely 
many isolated eigenvalues on the vertical line $\{\lambda\in \C:$ 
Re$\,\lambda = \omega_{0}(U)\}$. The real parts of the remaining
eigenvalues are less than $\omega_{0}(U)$. 
On the other hand,   the spectrum of $G_\#$ consists of eigenvalues only,
because $D(G_\#)$ is compactly embedded in $L^2_\#( \R^{1+n},\nu)$
as we proved in \cite{GL07}. 
 
So, $G_\#$ has a spectral gap 
that corresponds precisely to the asymptotic behavior  
of $(\cP^\#_\tau(I - \Pi))_{\tau\geq0}$. This implies that for each 
$\lambda$ with real part in $(0, +\infty)$, in $(\omega_{0}(U),0)$, and also for 
$\lambda \in i\R \setminus  2\pi i\Z /T$, for each $h\in L^2_\#( \R^{1+n},\nu)$
equation \eqref{e7} has a unique solution $u\in D(G_\#)$. For $\lambda =0$, it is easy to see that the 
range of $G_\#$ consists of the functions $h$ such that the mean 
value $\int_{0}^{T}\int_{\R^{n}}h(t,x)\dd \nu_{t}\dd t $ vanishes, 
and in 
this case the solution of \eqref{e7}  is unique up to constants. 

If the data are not periodic but just bounded, a natural Hilbert setting 
for problem \eqref{e7} is the space $L^2 (\R^{1+n},\nu)$ consisting of 
the Lebesgue measurable functions $h$ such that  the norm
$$\|h\|_{L^2 ( \R^{1+n},\nu)} = \bigg( \int_{\R}  \int_{\R^{n}} 
|h(s,x)|^{2}\dd \nu_{s} \ddc s\bigg)^{1/2}$$
is finite. A maximal regularity result similar to the one in the 
periodic space still holds, namely 
 if $\lambda \in \C$, $h\in L^2 ( \R^{1+n},\nu)$
and $u\in L^2 ( \R^{1+n},\nu) \cap H^{1,2}_{loc}(\R^{1+n}, \ddc t\times \ddc x)$ is a  
 solution of \eqref{e7}, then $u\in H^{1,2} ( 
\R^{1+n},\nu)$ i.e. $u_{t}$ and all the space derivatives $u_{x_{i}x_{j}}$ 
belong to $L^2 ( \R^{1+n},\nu)$. The operator
$$\left\{ \begin{array}{l}
G : D(G )=  H^{1,2} (\R^{1+n},\nu)\mapsto L^2 ( 
\R^{1+n},\nu), 
\\
\\
G u(s,x) =  u_s(s,x) + \cK(s)u(s,x) 
\end{array}\right. $$
is the infinitesimal generator of the evolution semigroup $ 
(\cP _\tau )_{\tau\geq0}$ in $L^2 ( \R^{1+n},\nu)$, defined as $\cP^\#_\tau$ by
\begin{equation}
    \label{rep2}
 ( \cP_\tau u  )(s,x)  = \left( P_{s,s+\tau}
    u(s+\tau,\cdot) \right)(x),\quad s\in\R,\ x\in\R^n,\
    \tau\geq 0 ,\; u \in L^2 ( \R^{1+n},\nu).
\end{equation}
See   \cite{GL07}. 
A part of the  properties of $(\cP^\#_\tau)_{\tau\geq 0}$ and $G_\#$ are enjoyed by  $(\cP_\tau 
)_{\tau\geq 0}$ and  $G$. 
However, without periodicity and compact embeddings, the results are less 
precise. 
$G$  has still a spectral gap: its spectrum contains the 
whole imaginary axis, and it has no elements with real part in 
$(c_{0},0)$, where $c_{0}<0$ depends on $A$ and $B$. Therefore, for each 
$\lambda $ with real part in $(c_{0},0)\cup (0, +\infty)$ and for each 
$h\in L^2(\R^{1+n},\nu)$, equation \eqref{e7} has a unique solution 
in $D(G)$. For $\lambda =0$, we  show that for $h \in 
L^2(\R^{1+n},\nu)$ problem \eqref{e7} has a solution 
in  $D(G)$ iff the 
function $t\mapsto M_{t}h$ has a primitive in $L^2(\R; \ddc t)$, in 
this case the solution is unique. 

The projection $\Pi$ defined in \eqref{pi} is still the spectral projection relative to the imaginary axis, 
the range of $\Pi$ is isomorphic to $L^2(\R; \ddc t)$, the 
restriction of 
$(\cP_\tau )_{\tau\geq 0}$ to the range of $\Pi$ is the translation semigroup in $L^2(\R; \ddc t)$, 
and the growth bound of $(\cP_\tau(I - \Pi))_{\tau\geq 0}$ does not 
exceed $ c_{0}$. 
So, for each $u\in L^2(\R^{1+n},\nu)$, $\cP_\tau u$ converges exponentially to $\Pi u$ 
as $\tau \to \infty$ and we have an estimate for the convergence 
rate; the optimal convergence rate is still an open problem. 

Our procedure is reversed with respect to the periodic setting. As a 
first result we show that $\cP_\tau(I - \Pi)$ converges exponentially to zero
through  Poincar\'e type inequalities that hold in $D(G)$. Then 
from the general theory of semigroups, it follows that the 
spectrum of the part of $G$ in $(I-\Pi)(L^2( \R^{1+n},\nu))$ is 
contained in the halfplane $ \{\lambda \in \C:$ Re$\,\lambda \leq c_{0}\}$. 
Moreover we obtain asymptotic behavior properties of $P_{s,t}$ from 
the asymptotic behavior properties of $\cP_\tau$, adapting to our 
situation the method used for the standard evolution semigroups and 
evolution operators. A crucial point in the proof is the continuity of the function 
$s\mapsto \|P_{s,s+\tau}\varphi \|_{L^{2}(\R^{n}, 
\nu_{s})}^{2}$, for each $\tau > 0$ and for each good $\varphi$, say 
$\varphi \in 
C^{1}_{b}(\R^{n})$. Eventually, we obtain
$$\|P_{s,t}(\varphi - M_{t}\varphi)\|_{L^{2}(\R^{n}, \nu_{s})} \leq 
  e^{ c_{0}(t-s)}\|\varphi \|_{L^{2}(\R^{n}, \nu_{t})} , \quad s<t, \; \varphi\in L^{2}(\R^{n}, 
\nu_{t}),$$
where $c_{0}$ is the above constant. 

In the last section we show that $(P_{s,t})_{s\leq t}$ is
hypercontractive, i.e.   $ P_{s,t}$ maps $L^{q}(\R^{n}, \nu_{t})$ 
into $L^{p(s,t)}(\R^{n}, \nu_{s})$ for suitable $p(s,t)>q$ if $s<t$, $q>1$, and
\begin{align*}
\| P_{s,t}\varphi\|_{L^{p(s,t)}(\R^n,\nu_{s})}\leq
\|\varphi\|_{L^q(\R^n,\nu_t)},\quad \varphi\in L^q(\R^n,\nu_t),\	s\leq t. 
\end{align*}
Moreover, $p(s,t)\geq 1+ (q-1)e^{2c_{0}(s-t)}$. Estimates of this type are
well-known in the autonomous case, see \cite{CMG96,Fuh98,Gro75}. As far as we know, this is the first hypercontractivity result in the nonautonomous case. 

Our approach is based on the ideas used in \cite{Gro75}. More precisely,
we differentiate 
\begin{align*}
\alpha(s)=\|P_{s,t}\varphi  \|_{L^{p(s,t)}(\R^n,\nu_s)}
\end{align*}
with respect to $s$ for suitable functions $\varphi$ and we show that $
\alpha'(s)\geq 0$ for $s\leq t$ with help of a variant of the classical logarithmic
Sobolev inequalities. 
The difference with the autonomous case is that we have to deal
with additional terms since the measure $\nu_s$ depends on $s$ as well.


\section{Spectral properties and asymptotic behavior}


In this section we investigate the spectrum of $(\cP_\tau)_{\tau\geq0}$
and $(\cP^\#_\tau)_{\tau\geq0}$, and of their generators. This leads 
to results about asymptotic behavior of such semigroups, and of the
evolution operator $P_{s,t}$. 

We already remarked that the general theory of parabolic evolution 
operators in Banach spaces cannot be directly applied to our $P_{s,t}$ 
because it does not act on a fixed $L^{2}$ space but it maps $X(t) = L^2(\R^n,\nu_t)$ into $X(s) = L^2(\R^n,\nu_s)$ 
and these spaces do not coincide in general. The same difficulty 
arises for the evolution semigroups $(\cP^\#_\tau)_{\tau\geq 0}$ and 
$(\cP _\tau)_{\tau\geq 0}$, since the general theory (see e.g. the 
monograph
\cite{CL99}) has been developed for 
evolution semigroups associated to evolution operators in a fixed 
Banach space $X$. Therefore, we have to start from the very 
beginning. However, some results can be extended to our situation with minor modifications. 
This is the case of the spectral mapping theorems of the next 
subsection.

\subsection{Spectral mapping theorems.}
We start with the spectral mapping theorem for $(\cP^\#_\tau)_{\tau\geq 0}$. 
Next proposition \ref{spectral} is is a variant  of \cite[Theorem 3.13]{CL99} for time-depending spaces.
 Its proof is based on the ``change-of-variable'' trick, see \cite{LMS95}.

 We need some preparatory remarks. 
 
 If $X$ is any Banach space, we define the space $L_\#^2(\R,X)$ as 
 the space of all Bochner measurable functions  $Z:\R \mapsto X$, 
 such that $Z(\theta +T) = Z(\theta)$ for almost  all $\theta\in \R$ and 
 $\|Z\|^{2} :=  \int_{0}^{T}\|Z(\theta)\|_{X}^{2}\dd
 \theta <\infty.$

If $X= L_\#^2( \R^{1+n})$, then $L_\#^2(\R,X)$ may be identified 
(setting $z(\theta,t,x) = Z(\theta)(t,x)$ for each $Z\in L_\#^2(\R,X)$)
with the space $L_\#^2(\R^{2+n})$ 
consisting of the Lebesgue measurable functions $z $ defined in 
$\R^{2+n}$ 
such that $z(\theta +T, t, x) = z(\theta  , t, x)$, $z(\theta , t+T, 
x) = z(\theta  , t, x)$ for almost all $\theta$, $t\in \R$ and $x\in 
\R^{n}$, endowed with the norm 
$$\| z \| = \frac{1}{T} \left(  \int_{0}^{T} \int_{0}^{T}\int_{\R^{n}} 
|z(\theta, t, x)|^{2} \nu_{t}(\ddc x)\dd t \dd \theta \right)^{1/2}. $$

 \begin{proposition}
  \label{spectral}
If $A$, $B$, and $f$ are $T$-periodic, then
\begin{align*}
  \sigma(\cP^\#_\tau)\setminus\{0\} = {\rm e}^{\tau\sigma(G_\#)},\quad\tau>0.
\end{align*}
 \end{proposition}
\begin{proof}
The inclusion ${\rm e}^{\tau\sigma(G_\#)}\subset    
\sigma(\cP^\#_\tau)$ comes from the general theory of semigroups, see 
e.g. \cite[\S 3.6]{EN00}. We 
have to prove that   $\sigma(\cP^\#_\tau)\setminus\{0\} \subset {\rm 
e}^{\tau\sigma(G_\#)}$, or,  equivalently, that if $\lambda \in \rho  
(G_\#)$ then $e^{\tau \lambda }\in \rho(\cP^\#_\tau)$. 

Set $X:= L_\#^2(\R^{1+n})$. We define two semigroups in the space $L_\#^2(\R,X)$. 
The first one is the 
$T$-periodic evolution 
semigroup associated to our semigroup  $ \cP^\#_\tau  $, the 
second one   is the so called multiplication semigroup by $ \cP^\#_\tau  $:
$$\begin{array}{lll}
(\widetilde{\cP}_\tau Z)(\theta)  & = & \cP^\#_\tau (Z(\theta -\tau )) , \quad \tau >0,
\\
\\
({\mathcal E}_\tau Z) (\theta ) & = & \cP^\#_\tau (Z(\theta )),  \quad \tau >0.
\end{array}$$
It is easy to see that the 
infinitesimal generator $A$ of $({\mathcal E}_\tau  )_{\tau \geq 0}$ is 
the multiplication operator by $G_\#$, that is 
$$D(A) = \{Z\in L_\#^2(\R,X):\;Z(\theta)\in 
D(G_\#) \;{\rm a.e}\}, \quad AZ(\theta) = G_\#Z(\theta),$$
the resolvent set $\rho(A)$ of $A$ coincides with $\rho(G_\#)$, and 
$(R(\lambda,A)F) (\theta) = R(\lambda, G_\#)(F(\theta))$ for all 
$\lambda \in \rho(G_\#)$, $F\in 
L_\#^2(\R,X)$ and $\theta \in \R$. 

Now we prove that $\rho(A) = \rho(\widetilde{G} )$, where $\widetilde{G}$ 
is the infinitesimal generator of $(\widetilde{\cP}_\tau)_{\tau\geq 0}$. 

Setting as above $z(\theta ,t,x) = Z(\theta)(t,x)$, we identify  $L_\#^2(\R,X)$ with 
$L_\#^2(\R^{2+n})$. Then  $\widetilde{\cP}_\tau$ and ${\mathcal E}_\tau z$ 
may be rewritten as semigroups in $L_\#^2(\R^{2+n})$,
$$\begin{array}{lllll}
(\widetilde{\cP}_\tau z)(\theta, t, x) & = & \cP^\#_\tau z(\theta - 
\tau, \cdot, \cdot)(t,x) & = & P_{t,t+\tau}z(\theta -\tau , t+\tau, \cdot)(x), 
\\
\\
({\mathcal E}_\tau z ) (\theta, t, x) & = &  \cP^\#_\tau z(\theta ,
\cdot, \cdot)(t,x) & = & P_{t,t+\tau}z(\theta, t+\tau ,\cdot)(x).
\end{array}$$
We define the isometry $J: L_\#^2(\R^{2+n})\mapsto  L_\#^2(\R^{2+n})$ by 
$$(Jz)(\theta, t, x) = z(\theta -t, t, x), \quad (\theta, t, x)\in \R^{2+n}. $$
Then ${\mathcal E}_\tau J = J\widetilde{\cP}_\tau$ for each $\tau >0$, 
and this implies immediately that $D(\widetilde{G}) = J^{-1}(D(A))$, $ 
\widetilde{G} = J^{-1}AJ$ and $\rho (\widetilde{G}) = \rho(A)$. 

So, we have 
$$\rho(G_\#) = \rho(A) = \rho (\widetilde{G}).$$
Since  $(\widetilde{\cP}_\tau)_{\tau\geq 0}$ is an evolution 
semigroup, then by the general theory of evolution semigroups we have 
$\rho(\cP^\#_\tau ) = \rho(\widetilde{\cP}_\tau ) ={\rm 
e}^{\tau\rho(\widetilde{G} )}$ for   each 
$\tau\geq 0$, see e.g.  \cite[Theorem 2.30]{CL99}.
In particular,  if $\lambda \in \rho(G_\#)$ then ${\rm e}^{\tau 
\lambda}\in \rho(\widetilde{\cP}_\tau ) = \rho(\cP^\#_\tau )$, and the 
statement follows. 
 \end{proof}
 
We have a corresponding result in the non-periodic case.   
The proof is the same, with the space $L ^2(\R,L ^2( \R^{1+n},\nu))$ instead of
$L_\#^2(\R,L_\#^2(\R^{1+n}))$.

 \begin{proposition}
  We have 
  \begin{align*}
   \sigma(\cP_\tau)\setminus\{0\}={\rm e}^{\tau\sigma(G)},\quad\tau>0.
  \end{align*}
 \end{proposition}


\subsection{Exponential dichotomy and asymptotic behavior of  $P_{s,t}$ in the periodic case.}

Throughout this section we assume that $A$, $B$, and $f$ are 
$T$-periodic. 
As in the case of a fixed Banach space $X$ (see \cite{Hen81}), the asymptotic behavior of $P_{s,t}$
is determined by the spectral properties of the Poincar\'e operators, 
$$V(t) := P_{t-T,t} \in \cL(L^2(\R^n,\nu_t)), \quad t\in \R.$$

In the following proposition we collect the spectral properties of 
the operators $V(t)$ that will be used in the sequel. An important 
role is played by the projections on the subspace of constant 
functions, given by the mean values:
\begin{equation}
    \label{Mt}
    M_{t}\varphi :=  \int_{\R^{n}}\varphi  \dd  
\nu_{t}, \quad \varphi\in L^2(\R^n,\nu_t). 
\end{equation}

We recall that the eigenvalues of $U(t+T,t)$ are independent of $t$, 
and that $\lambda $ is a semisimple eigenvalue of $U(t+T, t)$ iff it 
is a semisimple eigenvalue of $U(T,0)$ iff it is a semisimple 
eigenvalue of $U^{*}(T,0)$. Moreover, denoting by $r_{0}$ the spectral radius of all the operators 
$U(t+T,t)$ we have $ \omega_{0 }(U) =\frac{1}{T} \log r_{0}$, i.e. 
$$r_{0} = e^{\omega_{0}(U)T}.$$

\begin{proposition}
\label{evolutionspectrum}
The spectrum of $V(t)$ is independent of $t$, and it consists of
isolated eigenvalues with modulus $\leq 1$, plus $0$. 
Moreover, 
\begin{enumerate}
\item If $\lambda \in \sigma (V(t))$ and 
  $|\lambda|=1$, then $\lambda=1$, it is a simple eigenvalue, and the 
  eigenspace consists of the constant functions. The spectral projection 
  is $M_{t}$. 
\item If $\lambda \in \sigma (V(t))$ and  $|\lambda|<1$, then 
$|\lambda|\leq r_{0}$, and the generalized eigenspace consists of 
polynomials with degree $\leq \frac{\log |\lambda|}{\log r_{0}}$.
\item For $|\lambda |<1$, there exists a non-constant polynomial $\varphi$ of degree $1$
 satisfying $V(t)\varphi =\lambda \varphi $ if and only if 
  $\lambda\in\sigma(U(T,0))$. In this case,  
 \begin{align*}
     \varphi (x) = \langle \vec c,  x\rangle 
     +\frac{1}{\lambda-1}\langle\vec c ,g(t ,t-T)\rangle,
 \end{align*}
 where $\vec c$ is an eigenvector of $U^{*}(t, t-T)$ with eigenvalue
 $\lambda$.
 \item An eigenvalue  of $V(t)$ with modulus equal to $r_{0}$
 is semisimple iff it is a semisimple eigenvalue of $U(T,0)$. 
\end{enumerate}
\end{proposition}
\begin{proof}
By  estimates \eqref{stime}, $V(t)$ maps continuously $L^2(\R^n,\nu_t)$ into 
$H^{1}(\R^n,\nu_t)$, which is compactly embedded in 
$L^2(\R^n,\nu_t)$ because $\nu_t$ is a Gaussian measure with 
nondegenerate covariance matrix. Therefore it is a  compact operator, and its
spectrum consists of $0$ and of isolated nonzero eigenvalues. 

From the equality
$$P_{s,t}V(t) = V(s)P_{s,t}, \quad s<t,$$
it follows that if $\varphi$ is an eigenfunction of $V(t)$ with eigenvalue $\lambda 
\neq 0$, then   $P_{s,t}\varphi $ is an eigenfunction of $V(s)$ with eigenvalue 
 $\lambda$. It follows that the spectrum of $V(t)$ is independent of $t$. 

 Let $\varphi$ be again  an eigenfunction of $V(t)$ with eigenvalue $\lambda 
 \neq 0$. 
 Then $P_{t-nT,t}\varphi = (V(t))^n \varphi = \lambda^n \varphi$ for 
 each $n\in \N$, so 
 that,  by estimate \eqref{stime},
\begin{equation}
    \label{autovalori}
|\lambda|^n\|D^\alpha \varphi\|_{L^2(\R^n,\nu_t)}\leq C{\rm
e}^{\omega|\alpha|nT}\|\varphi\|_{L^2(\R^n,\nu_t)},\quad n\in\N,
\end{equation}
for $\omega\in(\omega_0(U),0)$ and for each multi-index $\alpha$.  Therefore, $|\lambda |\leq 1$ and 
$D^\alpha \varphi=0$ if
$|\alpha|> \log |\lambda|/\omega_0(U)T$.  This proves that the 
eigenspace consists of polynomials with degree $\leq  \log |\lambda|/
\log r_{0}$.

To complete the proof of statement (b) we argue by recurrence. Assume 
that for some $r\in \N$ the kernel of $(\lambda I - V(t))^{r}$ 
consists of polynomials with degree $\leq  \log |\lambda|/\log r_0$, and let $\varphi \in$ Ker$\,(\lambda I - 
V(t))^{r+1}$. Then the function $\psi := \lambda \varphi - 
V(t)\varphi$ is a polynomial with degree $\leq \log |\lambda|/\log r_0$, as well as  $V(t)^{k}\psi$ for each $k\in 
\N$. Indeed, each $P_{s,t}$ maps polynomials of degree $n$ into 
polynomials of degree $\leq n$, for each $n\in \N$. Since
$$V(t)^{n}\varphi = \lambda^{n}\varphi - 
\sum_{k=0}^{n-1}\lambda^{n-1-k}V(t)^{k}\psi, \quad n\in \N,$$
then $D^{\alpha }(V(t)^{n}\varphi) =  \lambda^{n}D^{\alpha 
}\varphi $, for $|\alpha| > \ \log |\lambda|/\log r_0$. Using \eqref{autovalori} as before we 
see that $\varphi $ is a polynomial with degree $\leq  \log |\lambda|/\log r_0$. 
This proves statement (b). 

 \vspace{1mm}

Now we can prove statement  (a). 
Estimate \eqref{autovalori} shows that if 
$V(t)\varphi = \lambda\varphi$ and $|\lambda|=1$, then $\varphi$ is 
constant, and since $V(t)$ is the identity on constant functions, 
we have $\lambda=1$. By statement (b), also the kernel of $(I-V(t))^{2}$ consists 
of the constant 
functions, so that it coincides with the kernel of $I-V(t)$, and  $1$ is a simple eigenvalue.

The projection $M_{t}$ maps 
$L^2(\R^n,\nu_{t })$ onto the kernel of $I -V(t)$. Moreover, it 
commutes with $V(t)$, since for each $\varphi \in L^2(\R^n,\nu_{t })$ 
we have
\begin{align*}
    V(t)M_{t}\varphi & =  M_{t}\varphi = \displaystyle{\int_{\R^{n}} }
\varphi(x)   \nu_t(\ddc x) 
 =  \int_{\R^{n}} 
(V(t)\varphi)(x)   \nu_{t-T}(\ddc x)\\
&= \displaystyle{\int_{\R^{n}} }
(V(t)\varphi)(x)   \nu_{t}(\ddc x)
 =  M_{t}V(t)\varphi .
\end{align*}
Since $1$ is a simple eigenvalue, then $M_{t}$ is the associated spectral projection. 

\vspace{1mm}

Let us prove statement (c). 
Let $\varphi(x)=c+\langle \vec c, x\rangle $ with $c\in\C$ and $\vec
c\in\C^n$. Then
\begin{align*}
(V(t)\varphi)(x)=c+\langle \vec c ,U(t,t-T)x\rangle + \langle \vec 
c,g(t,t-T)\rangle.
\end{align*}
Hence, $V(t)\varphi=\lambda \varphi$ iff $\lambda\in\sigma(U(t,t-T))$,
$\vec c $ is an eigenvector of $U^{*}(t,t-T)$ with eigenvalue $\lambda$ and 
$c=\langle \vec c ,g(t,t-T)\rangle /(\lambda-1)$.

\vspace{1mm}

Note that $U^{*}(t,t-T)$ has at least one eigenvalue $\lambda$ with modulus 
equal to $r_{0}$. By statement (b), the corresponding generalized eigenspace  of $V(t)$ consists 
of  first order polynomials. Let  $\varphi(x) = c +\langle \vec c , 
x\rangle $ be a first order polynomial in the kernel of 
$\lambda I - V(t)$. The equation $(\lambda I -V(t))\psi = \varphi$ 
may be solved only by first order polynomials. 
If $\psi(x)  = c_{1}+\langle \vec c_{1}, x\rangle $, 
we have $(\lambda I -V(t))\psi = \varphi$ iff 
$$(\lambda -1)c_{1}+\langle \lambda \vec c_{1}, x\rangle - 
\langle  \vec c_{1}, U(t,t-T) x + g(t,t-T) \rangle  = 
c +\langle \vec c , x\rangle , \quad x\in \R^{n},$$
that is, $(\lambda -1)c_{1} -\langle \vec c_{1},g(t,t-T) \rangle = 
c $ and $  \lambda \vec c_{1}  -  U^{*}(t,t-T)\vec c_{1}   =   \vec c $. Since 
$\vec c $ is an eigenvector of $U^{*}(t,t-T)$ and 
$c =\langle \vec c  ,g(t,t-T)\rangle /(\lambda-1)$, we have 
$(\lambda I -V(t))\psi = \varphi$ iff $\vec  c_{1}\in $ Ker$\,(\lambda I 
-U^{*}(t,t-T))^{2}$ $\setminus $ Ker$\,(\lambda I 
-U^{*}(t,t-T)) $, and $c_{1} = (c + \langle  \vec 
c_{1},g(t,t-T)\rangle )/(\lambda -1)$. Statement (d) follows. 
\end{proof}

Statements (a) and (b) are a generalization to the periodic 
nonautonomous case of the results of 
\cite[Proposition 3.2]{MPP02} concerning the spectral properties of 
elliptic Ornstein-Uhlenbeck operators.

As a consequence of Proposition \ref{evolutionspectrum} we describe 
the asymptotic behavior of $P_{s,t}\varphi$ for each $\varphi\in L^2(\R^n,\nu_{t })$.

\begin{proposition}
\label{expdich}
\begin{itemize} 
\item[(i)] For each $\omega \in (\omega_{0}(U),0)$ there exists $M=M(\omega)$ such that 
\begin{equation}
    \label{asintotico}
    \|P_{s,t}(\varphi - M_{t}\varphi)\|_{L^{2}(\R^{n}, \nu_{s})} 
    \leq Me^{\omega(t-s)}\|\varphi\|_{L^{2}(\R^{n}, \nu_{t})} , \quad 
    s<t, \; \varphi\in L^2(\R^n,\nu_{t }). 
\end{equation}
\item[(ii)] For each $\omega < \omega_{0}(U)$ there is no $M $ such 
that   \eqref{asintotico} holds. 
\item[(iii)] Estimate \eqref{asintotico} holds  for $\omega = 
\omega_{0}(U)$ iff all the eigenvalues of $U(T,0)$ with modulus equal 
to $r_{0}$ are semisimple.
\end{itemize}
\end{proposition}
\begin{proof}
(i) Let us split $L^2(\R^n,\nu_{t })$ as the direct sum 
$L^2(\R^n,\nu_{t }) = X_{t} \oplus X_{c}$, where $X_{t}$ consists of 
the functions with zero mean value and $X_{c}$ consists of the 
constant functions.  The orthogonal projection on $X_{c}$ is  
$M_{t}$, and by Proposition \ref{evolutionspectrum} (a)
it coincides with the spectral projection associated to the 
eigenvalue $1$ of $V(t)$. The spectral radius of the part of $V(t)$ 
in $X_{t}$ does not exceed $r_{0}$ by Proposition 
\ref{evolutionspectrum} (b), but in fact it is equal to 
$r_{0}$, because for each $\lambda \in  \sigma(U(T,0)) $ 
with modulus $r_{0}$, $\lambda $ is also an eigenvalue of 
$V(t)$ by Proposition \ref{evolutionspectrum} (c). 

From now on we can proceed as in the standard case of constant 
underlying space (e.g, 
\cite[\S 7.2]{Hen81}). For $\varphi \in X_{t}$ and 
$t-s>2T$ set $m= [s/T]+1$, $k= [t/T]$. Since $P_{mT,kT} = V(0)^{k-m}$, 
then 
$$\|P_{s,t}\varphi \|_{L^{2}(\R^{n}, \nu_{s})}  = 
\|P_{s,mT}V(0)^{k-m}P_{kT,t}\varphi \|_{L^{2}(\R^{n}, \nu_{s})}  \leq 
\|V(0)^{k-m}\|_{\LL (X_{0})}\|\varphi \|_{L^{2}(\R^{n}, \nu_{t})},$$ 
where $(k-m)T \geq t-s-2T$. 
Since $\lim _{h\to \infty}\|V(0)^{h}\|_{\LL (X_{0})} = r_{0} =
e^{\omega_{0}(U)T}$, it follows that for each $\omega >\omega_{0}(U)$ 
there exists $M=M(\omega)$ such that 
$$\|P_{s,t}\varphi \|_{L^{2}(\R^{n},\nu_{s})}\leq Me^{\omega(t-s)} \|\varphi 
\|_{L^{2}(\R^{n}, \nu_{t})}, \quad s<t ,$$
which is \eqref{asintotico} in our case, because $M_{t}\varphi=0$. 

For general $\varphi \in L^{2}(\R^{n}, \nu_{t})$, applying the above 
estimate to $\varphi - M_{t}\varphi$ gives 
\eqref{asintotico}.  

\vspace{1mm}
\noindent (ii) By Proposition \ref{evolutionspectrum}, $V(t)$ has 
some eigenvalue $\lambda$ with modulus $r_{0}$. If 
$\varphi$ is an eigenfunction, then it belongs to $X_{t}$ so that 
$M_{t}\varphi=0$. Moreover, for $s = t-kT$ we have 
$P_{s,t}\varphi = \lambda^{k}\varphi $ so that 
$\|P_{s,t}(\varphi - M_{t}\varphi)\|_{L^{2}(\R^{n}, \nu_{s})} 
= \| P_{s,t}\varphi \|_{L^{2}(\R^{n}, \nu_{s})}  = 
e^{\omega_{0}(U)(t-s)}$, and (ii) follows. 

\vspace{1mm}
\noindent (iii)
By proposition \ref{evolutionspectrum}(b)(c), the eigenvalues of 
$V(t)$ with  modulus in
$(r_{0}^{2},0)$ coincide with the eigenvalues of $U(t+T,t)$ 
with  modulus in $(r_{0}^{2},0)$. 
Therefore, setting $r_{1}= \max\{|\lambda|:$ $\lambda \in \sigma (U(T,0)),$ 
$|\lambda|< r_{0}\}$, $V(t)$ has no eigenvalues 
 with modulus in $(\max 
 \{r_{1},r_{0}^{2}\},r_{0})$, while the part of the 
spectrum of $V(t)$ with modulus equal to $r_{0}$ consists 
of eigenvalues of $U( T, 0)$. 
 Let $Q_{t}$ be the associated 
spectral projection, and let us further decompose $X_{t}$ as the 
direct sum   $Q_{t}(X_{t})\oplus (I-Q_{t})(X_{t})$. Note that for 
$s<t$, $P_{s,t}$ maps 
$Q_{t}(X_{t})$ into $Q_{s}(X_{s})$ and $(I-Q_{t})(X_{t})$ into 
$(I-Q_{s})(X_{s})$. 
The spectral radius of 
$V(0)(I-Q_{0}-M_{0})$ does not exceed $\max 
 \{r_{1},r_{0}^{2}\}$, so that arguing as in the proof of statement (i) we obtain that 
 for each $\omega \in (\log  \max 
 \{r_{1},r_{0}^{2}\}, \omega_{0})$
there is $M>0$ such that
$$ \|P_{s,t}(I-Q_{0} - M_{0}) \|_{\mathcal{L}(L^{2}(\R^{n}, \nu_{t}), L^{2}(\R^{n}, \nu_{s}))} 
\leq M e^{\omega (t-s)}, \quad s<t.$$
Assume that all the  eigenvalues of $U(T,0)$ with modulus $r_{0}$ are 
semisimple. Then  
by Proposition \ref{evolutionspectrum}(d) they are  semisimple eigenvalues 
of $V(0)$. Therefore there is $C>0$ such that 
$$ \|V(0)^{k} Q_{0}\|_{\mathcal{L}(L^{2}(\R^{n}, \nu_{0}))} 
\leq C r_{0}^{k}, \quad k\in \N.$$
Arguing again as in the proof of statement (i), we obtain that 
\eqref{asintotico} holds also with $\omega = \omega_{0}(U)$. 

If one of the eigenvalues $\lambda$ of $U(T,0)$ with modulus $r_{0}$
is not semisimple, again by proposition \ref{evolutionspectrum}(d) it  is a 
non-semisimple eigenvalue of $V(t)$. Then there are nonzero 
functions $\varphi_{0}$, $\psi_{0} \in X_{t}$ such that $(\lambda I-V(t))\varphi_{0} = 
\psi$, $(\lambda I-V(t))\psi _{0}=0$. It follows that $V(t)^{k}\varphi_{0} = 
\lambda^{k}\varphi_{0} - k\psi_{0}$, for each $k\in \N$. Arguing as in the 
proof of statement (ii) we see that \eqref{asintotico} cannot hold 
for $\varphi = \varphi_{0}$ and $\omega = \omega_{0}(U)$. 
\end{proof}

Proposition \ref{expdich} establishes a sort of exponential dichotomy 
with any exponent $\omega \in (\omega_{0}(U),0)$ for $P_{s,t}$. 
Indeed, the  projections 
$$\varphi\mapsto  M_{t}\varphi , \quad t\in \R,$$
map each $L^{2}(\R^{n}, \nu_{t})$ into the common one-dimensional subspace 
$X_{c}$ of the constant functions, and satisfy

\begin{itemize}
    \item[(a)] $M_{s}P_{s,t} = P_{s,t}M_{t}$, for $ s<t$; 
    \item[(b)] $P_{s,t}:$ Range $M_{t} \mapsto$ Range $M_{s}$ is 
    invertible (in fact, it is the identity in $X_{c}$); 
    \item[(c)] $\| P_{s,t}(I -M_{t})\|_{\LL(L^{2}(\R^{n}, \nu_{t}), 
    L^{2}(\R^{n}, \nu_{s})) } \leq Me^{\omega(t-s)}$, $s<t$. 
\end{itemize}


\subsection{Spectral gap of $G_\#$ and  asymptotic behavior of $(\cP_\tau^{\#})_{\tau\geq0}$.}

 Since $D(G_\#)$ is compactly embedded in $L^2_\#((0,T)\times\R^n,\nu)$, see 
\cite{GL07}, the spectrum of $G_\#$ contains eigenvalues only. This allows us to do
 further investigations of the spectrum of $G_\#$.

The next proposition shows that all the generalized eigenfunctions of 
$G_\#$ have a special structure.

\begin{proposition}
    \label{polinomi}
Assume that $u\in D(G_\#^r)$ satisfies $(\lambda I-G_\#)^ru=0$ for some
$\lambda\in\C$ and some $r\in\N$. Then
\begin{align*}
u(t,x)=\sum\limits_{|\alpha|\leq K}c_\alpha(t)x^\alpha,
\end{align*}
where $K\leq \frac{\RE\lambda}{\omega_0(U)}$ and $c_\alpha\in
H^1_\#(0,T)$.
\end{proposition}
\begin{proof}
Let us start with $r=1$.
Since $G_{\#}u=\lambda u$, we have $\cP^\#_\tau u={\rm
e}^{\lambda\tau}u$ for $\tau\geq0$. Therefore, by
estimates \eqref{stime}, for any $\omega>\omega_0(U)$ there
exists $C>0$, such that for any multi-index $\alpha$,
\begin{align*}
\|{\rm e}^{\lambda\tau}D^\alpha _{x}u\|_{L^2_\#( \R^{1+n},\nu)}=
\|D^\alpha _{x}\cP_\tau^\#u\|_{L^2_\#( \R^{1+n},\nu)}
\leq C {\rm e}^{\omega|\alpha|\tau}\|u\|_{L^2_\#( \R^{1+n},\nu)},\quad \tau\geq 1.
\end{align*}
Letting $\tau\to\infty$, we obtain
\begin{align*}
\|D^\alpha _{x}u\|_{L^2_\#( \R^{1+n},\nu)}=0
\end{align*}
for $\RE\lambda>\omega|\alpha|$. This implies that $u(t,\cdot)$ is
a polynomial of degree less than or equal to $ |\RE\lambda|/\omega$
for any $ \omega \in (\omega_0(U),0)$. 

Suppose now that the assertion holds
for $r= 1, \ldots, r_{0}$ and assume that $u\in D(G_\# ^{r_{0}+1})$ 
satisfies
$(\lambda I -G_\#)^{r_{0}+1}u=0$ for some $\lambda\in\C$. Then,
\begin{align*}
\cP^\#_\tau u={\rm e}^{\lambda
\tau}\sum\limits_{j=0}^{r_{0}}\frac{\tau^j}{j!}(\lambda-G_\#)^ju,\quad
\tau\geq1.
\end{align*}
By the induction hypothesis, $(\lambda I -G_\#)^ju$ is a polynomial of 
degree $\leq  \RE\lambda /\omega_0(U) $, so that  $D^{\alpha}(\lambda 
I -G_\#)^ju=0$ for 
$1\leq j\leq r_{0}$ and $|\alpha|> \RE\lambda/\omega_0(U)$. So, we obtain
\begin{align*}
D^\alpha_{x}\cP^{\#}_\tau u={\rm e}^{\lambda 
\tau}D^\alpha_{x}u,\quad\tau\geq 1,
\end{align*}
and the assertion for $r_{0}+1$ follows as above.
\end{proof}

Proposition \ref{polinomi} implies that the eigenfunctions with  
eigenvalues $\lambda$ such that Re$\,\lambda \in (2\omega_{0}(U),0]$ 
are first or zero order polynomials with respect to $x$, with 
coefficients possibly depending on $t$. In the next 
proposition we characterize  the eigenvalues that have eigenfunctions 
of this type.

\begin{proposition}
\label{eigen}
Assume that $u(t,x)=c(t)+\sum_{i=1}^nc_i(t)x_i$ with $c$, $c_i\in
H^1_{\#}(0,T)\setminus \{0\}$ satisfies $G_\# u=\lambda u$ for some $\lambda\in\C$.
Then 
\begin{equation}
    \label{lambda}
\lambda\in\left( \frac1T\log\sigma(U(T,0))+\frac{2\pi i}T\Z\right)\cup\frac{2\pi
i}T\Z.
\end{equation}
Conversely, for each $\lambda$ satisfying \eqref{lambda} there is a 
function $u\neq 0$ as above such that $G_\# u=\lambda u$. 
\end{proposition}
\begin{proof}
Since $u$ satisfies $G_\# u=\lambda u$, we have
\begin{align}
c'(t) & =\lambda c(t)-\langle f(t),\vec c(t)\rangle ,\quad t\in\R,
\label{eigen1}
\\
c(0) & = c (T)
\label{eigen2}
\\
\vec c \,'(t) & = (\lambda-A^{*}(t))\vec c(t),\quad t\in\R
\label{eigen3}
\\
\vec c(0) & = \vec c (T)
\label{eigen4}
\end{align}
where $\vec c=(c_1,\dots,c_n)^T$. Note that every solution of
\eqref{eigen3} is of the form
\begin{equation}
    \label{c(t)}
\vec c(t) ={\rm e}^{\lambda t}U^*(0,t)\vec c_0\mbox{ with }\vec
c_0\in\C^n.
\end{equation}
If $\vec c_0=0$ we have $\vec c(t)\equiv 0$ for $t\in\R$, Hence,  
the solutions of \eqref{eigen1} are given by $c(t)={\rm
e}^{\lambda t}c_0$ with any $c_{0}\in \C$, and equation \eqref{eigen2} can be satisfied
iff $\lambda\in\frac{2\pi i}T\Z$.

If  $\vec c_0\neq0$, $\vec c(t)$ satisfies \eqref{eigen4} iff $\vec c_0$ is an eigenvector of $V^{*}(0)$ with 
eigenvalue $e^{-\lambda T}$, i.e. iff
\begin{align}
\lambda\in -\frac1T\log\sigma(U^*(0,T))+\frac{2\pi i}T\Z = 
\frac1T\log\sigma(U(T,0))+\frac{2\pi i}T\Z.
\end{align}
Moreover, since all the solutions of \eqref{eigen1} are given by
\begin{equation}
c(t)={\rm e}^{\lambda t}c_0-\int\limits_0^t{\rm
e}^{\lambda(t-s)}\langle f(s),\vec c(s)\rangle \dd s\mbox{ with
}c_0\in\C ,
\label{soleigen1}
\end{equation}
and ${\rm e}^{\lambda t}\neq 1$ for $\lambda\in
-\frac1T\log\sigma(U^*(0,T))+\frac{2\pi i}T\Z$, we can find
$c_0\in\C$ such that the function given by \eqref{soleigen1}  is a solution to
\eqref{eigen2}.  
\end{proof}

\begin{corollary}
\label{Cor:eigen}
\begin{itemize}
    \item[(i)] $\sigma(G_\#)\cup i\R = \frac{2\pi i}T\Z$; for 
    each $k\in \Z$ the eigenvalue $\frac{2\pi ik}T $ is simple and the 
    eigenspace is spanned by $u(t,x) := e^{2\pi ikt /T}$. 
    \item[(ii)] The strips $\{\lambda\in \C :\;\mbox{\rm 
    Re}\,\lambda\in (\omega_{0}(U),0)\}$ and $\{\lambda\in \C :\;\mbox{\rm 
    Re}\,\lambda\in (a,\omega_{0}(U))\}$ are contained in 
		$\rho(G_\#)$. Here $a = \max\{ 2\omega_{0}(U), \frac{1}{T}\log 
    |\mu|:\; \mu\in \sigma (U(T,0)), \, |\mu|< e^{\omega_{0}(U)T}\}$. 
    \item[(iii)] $\lambda \in \sigma(G_\#)$ and $\RE \,\lambda =
    \omega_{0}(U)$ iff $\mu :=e^{\lambda T} \in \sigma (U(T,0))$ and 
    $|\mu| = \omega_{0}(U)$; for 
    each $k\in \Z$ the eigenvalue $\lambda +\frac{2\pi ik}T $ is 
    semisimple iff $e^{\lambda T}$ is a semisimple eigenvalue of $U(T,0)$.
\end{itemize} 
\end{corollary}
\begin{proof} All the claims are immediate consequences of 
Propositions \ref{polinomi} and \ref{eigen}, except the statements about 
semi-simplicity. 

Let $\lambda = 2\pi ik/T$, and let $\varphi\in $ Ker $(\lambda I - 
G_\#)^{2}$, i.e. 
$(\lambda I - G_\#)\varphi (t,x)= c e^{2\pi ikt /T}$ for some $c\in \R $. 
By proposition \ref{polinomi}, $\varphi = \varphi(t)$ is independent of 
$x$, and $G_\# \varphi (t,x)= \varphi'(t)$, so that $\varphi(t) = e^{2\pi 
ikt /T}(\varphi(0) -ct)$; since  $\varphi$ is $T$-periodic then $c=0$. 
Therefore, the kernel of $(\lambda I - G_\#)^{2}$ is equal to the kernel of $ \lambda I - 
G_\#$. 

Let now $\lambda$ be an eigenvalue with real part equal to 
$\omega_{0}(U)$. By Proposition \ref{polinomi}, all the generalized 
eigenfunctions $v$ are first order polynomials with respect to $x$. 

So, let  $v(t,x)=c_1(t) + \langle \vec c_1(t),x\rangle$
satisfy $(\lambda I-G_\#)v=u$, where $u(t,x)=c_2(t)+\langle \vec
c_2(t),x\rangle$ is an eigenfunction with eigenvalue $\lambda$.
 Note that $\vec c_2\neq0$. As in the proof of Proposition \ref{eigen}, we obtain
\begin{align}
c_1'(t)&=\lambda c_1(t)-<f(t),\vec c_1(t)>-c_2(t),\quad
t\in\R,
\label{eigenb1}
\\
c_1(0)&=c_{1}(T)
\label{eigenb2}
\\
\vec c_1\,'(t)&=(\lambda-A^*(t))\vec c_1-\vec c_2(t),\quad t\in\R
\label{eigenb3}
\\
\vec c_1(0)&=\vec c_{1}(T)
\label{eigenb4}
\end{align}
All the solutions of \eqref{eigenb3} are of the form
\begin{align*}
\vec c_1(t)={\rm e}^{\lambda t}U^*(0,t)\vec c_{1,0}-\int\limits_0^t{\rm
e}^{\lambda(t-s)}U^*(s,t)\vec c_2(s)\dd s
\end{align*}
with some $\vec c_{1,0}\in\C^n$. Since $u$ is an eigenfunction, the 
  proof of Proposition \ref{eigen} yields 
$\vec c_2(s)={\rm e}^{\lambda s}U^*(0,s)\vec c_{2,0}$   where $\vec 
c_{2,0}$ is some
eigenvector of ${\rm e}^{\lambda T}U^*(0,T)$ with  
eigenvalue $1$.
Hence,
\begin{align*}
\vec c_1(t)={\rm e}^{\lambda t}U^*(0,t)\vec c_{1,0}-t{\rm e}^{\lambda
t}U^*(0,t)\vec c_{2,0},\quad t\in\R,
\end{align*}
Therefore, \eqref{eigenb4} is satisfied iff $\vec c_{1,0}={\rm e}^{\lambda 
T}U^*(0,T)\vec c_{1,0}-T\vec c_{2,0}$, that is
\begin{equation}
    \label{eigenb5}
    (1-{\rm e}^{\lambda T}U^*(0,T))\vec c_{1,0}= -T\vec c_{2,0},
\end{equation}
so that  $\vec c_{1,0}$ belongs to the kernel of $(1-{\rm e}^{\lambda 
T}U^*(0,T))^{2}$. If ${\rm e}^{ \lambda T}$ is a semisimple eigenvalue 
of $U(T,0)$, then $1$ is a semisimple eigenvalue of ${\rm e}^{\lambda 
T}U^*(0,T)$, and the only couple $(\vec c_{1,0},\vec c_{2,0})$ that satisfies \eqref{eigenb5} is 
$(0,0)$, so that $v=u\equiv 0$. If ${\rm e}^{ \lambda T}$ is not semisimple, there are 
nonzero couples $(\vec c_{1,0},\vec c_{2,0})$ that satisfy \eqref{eigenb5}. 
Using such couples, 
nonzero solutions $c_{1}(t)$, $\vec c_{1}(t)$ of \eqref{eigenb1}, \ldots, 
\eqref{eigenb4} may be found, and the corresponding functions $v(t) = 
c_{1}(t) +  \langle \vec c_1(t),x\rangle$
satisfy $(\lambda I -G_\#)^{2}v=0$, $(\lambda I -G_\#) v\neq 0$.
\end{proof}

\begin{remark}
\label{solvability}
{\em The spectral projection of $G_\#$ corresponding to the 
eigenvalue $0$ is }
$$u\mapsto \frac{1}{T} \int_{0}^{T}\int_{\R^{n}}u(t,x)\ddc 
\nu_{t}\ddc t.$$
{\em Indeed, it maps $L^2_\#( \R^{1+n},\nu)$ onto  the kernel $X_{c}$ of 
$G_\#$ and it commutes with $G_\#$. This implies that for $h\in L^2_\#( \R^{1+n},\nu)$
the equation}
$$G_\#u = h$$
{\em has a solution $u\in D(G_\#)$ iff the mean value  $\int_{(0,T)\times 
\R^{n}}h(t,x)\ddc \nu$ vanishes, and in this case the solution is 
unique up to constants. }
\end{remark}

\begin{remark}
\label{autonomous}
{\em In the autonomous case $A(t)\equiv A$, $f(t)\equiv 0$, 
$B(t)\equiv B$ we have a complete characterization of the spectrum 
of $G_\#$, }
$$\sigma(G_\#) = \bigg\{\lambda \in \C: \; \lambda = \frac{2k\pi i}T +
\sum_{j=1}^rn_j\lambda_j ; \;k\in \Z, \, \ n_j\in\N  \cup \{0\} \bigg\}$$
{\em where $\lambda_j$, $j=1, \ldots , r$ are the eigenvalues of $A$.}

{\em Indeed, in this case our evolution system of measures consists of a 
unique measure  $\nu $  independent of $t$, which is the invariant measure of  the 
Ornstein-Uhlenbeck semigroup $T(t)$, 
and $G_\#$ may be seen as the closure of the sum of the 
resolvent-commuting operators} 
$$\left\{ \begin{array}{l}
G_{1}: D(G_{1}):= \{ u\in L^2_\#( \R^{1+n},\nu):\; \exists u_{t}\in 
L^2_\#( \R^{1+n},\nu)\}\mapsto L^2_\#( \R^{1+n},\nu)\},  
\\
\\
G_{1}u =u_{t}, \end{array}\right.$$

$$\left\{ \begin{array}{l}
G_{2}: D(G_{2}):= \{ u\in L^2_\#( \R^{1+n},\nu):\; \exists u_{x_{i} }, \,
u_{x_{i}x_{j}}\in 
L^2_\#( \R^{1+n},\nu)\}\mapsto L^2_\#( \R^{1+n},\nu)\},  
\\
\\
(G_{2}u) (t,x)= \cK  u(t, \cdot)(x),  \end{array}\right.$$
{\em hence its spectrum is the sum of the spectra of $G_{1}$ and of $G_{2}$. 
The spectrum of $G_{1}$ is easily seen to be $\frac{2\pi i}T\Z $, 
while the spectrum of $G_{2}$ is equal to the spectrum of the 
Ornstein-Uhlenbeck operator $\cK$ in $L^{2}(\R^{n}, \nu)$, that was 
characterized in \cite{MPP02} as the set of all the complex numbers of the type 
$\sum_{i=1}^rn_i\lambda_i$, where $\lambda_{i}$, $i=1, \ldots, r$ are 
the eigenvalues of $A$ and $n_i\in\N  \cup \{0\}$. }
\end{remark}

\begin{proposition}
\label{semispecproj}
We have
$$\KER (I -\cP_T^{\#})= L^2_\#(0,T) = \KER (I -\cP_T^{\#})^2 ,$$
so that  $1$ is a semisimple isolated eigenvalue of $\cP_T^{\#}$. The spectral 
projection $\Pi$ is given by 
$$\Pi u(t,x) := M_{t}u(t, \cdot), \quad t\in \R, \;x\in \R^{n}.$$
\end{proposition}
\begin{proof} 
 \cite[Corollary IV.3.8]{EN00} yields
\begin{align*}
\KER (I -\cP_T^{\#})=\overline{\KER(2\pi i\Z/T-G_\#)}^{L^2_\#(
(0,T)\times\R^n,\nu)}.
\end{align*}
Since $\KER( 2\pi ik/T-G_\#)$ is spanned by the function $u\mapsto {\rm e}^{2\pi
ik/T} $ for any $k\in\Z$ (see the proof of Proposition~\ref{eigen}),
the first equality follows.

Assume that $u\in\KER(I -\cP_T^{\#})^2$, i.e.
\begin{align*}
\left((I -P_{t,t+T})u(t+T,\cdot)\right)(x)=f(t),\quad\mbox{a.a. }t\in\R
\end{align*}
for some $f\in L^2_\#(0,T)$. By Proposition~\ref{evolutionspectrum},
$u(t)$ is independent of $x$ for a.a. $t\in\R$. Therefore,  $u\in L^2_\#(0,T) = \KER (Id-\cP_T)$. 
This means that  $1$ is a semisimple eigenvalue of $\cP_T^{\#}$. 

By Corollary \ref{Cor:eigen} and Proposition \ref{spectral}, there are no other 
eigenvalues with modulus greater than $e^{\omega_{0}(U)T}$, so that 
$1$ is isolated. 
The projection $u\mapsto \Pi u$ maps 
$L^2_\#((0,T)\times\R^n,\nu)$ onto  $L^2_\#(0,T)$ and it commutes 
with $\cP_T$. Since $1$ is a semisimple eigenvalue, it is the 
spectral projection. 
\end{proof}

\begin{corollary}
    \label{growthbound}
    The growth bound of $(\cP_\tau^{\#}(I-\Pi)_{\tau >0})$ is $\omega_0(U)$. In 
    other words, 
\begin{enumerate}
\item for $\omega>\omega_0(U)$ there exists $M>0$ such that
\begin{equation}
\label{expdecay}
\begin{array}{l}
    \displaystyle {\int\limits_0^T\int\limits_{\R^n}\left(
(\cP_\tau^{\#}(u-\Pi u))(t,x)\right)^2\nu_t(\dd x)\dd t\leq M
{\rm e}^{2\omega\tau}\int\limits_0^T\int\limits_{\R^n}
( (u-\Pi u)(t,x))^2\nu_t(\ddc x)\dd t, }
\\
\\
\quad u\in
L^2_\#(\R^{n+1},\nu),\ \tau\geq0 ;
\end{array}
\end{equation}
\item for $\omega<\omega_0(U)$ there does not exist any $M>0$ such that
\eqref{expdecay} holds.
\label{dec2}
\end{enumerate}
Moreover, estimate \eqref{expdecay} holds for $\omega = \omega_0(U)$ 
iff all the eigenvalues of $U(T,0)$  with modulus equal to 
$r_{0}$ are  semisimple.
\end{corollary}
\begin{proof}
Since  $\Pi u (t,x) = M_{t}u(t,\cdot)$, $t\in\R$, then the first 
assertion immediately follows from
Proposition~\ref{expdich}(i) and from the definition of $\cP_{\tau }^{\#}$. 

By Proposition~\ref{eigen},   $\log 
\sigma(U(T,0))/T \subset \sigma_{p}(G_\#)$, so that 
for any  $\mu \in \sigma(U(T,0))$ with modulus 
equal to ${\rm e}^{ \omega_0(U)T}$, there is a nonzero eigenfunction 
$u$ of $G$ with eigenvalue $\lambda = \log \mu /T$, 
such that $\|\cP_{\tau }^{\#}u \|_{L^2_\#((0,T)\times\R^n,\nu)}$ 
$=$ ${\rm e}^{ \omega_0(U)\tau} \|u \|_{L^2_\#((0,T)\times\R^n,\nu)}$  for 
each $\tau >0$. Hence, statement (b) holds.

If all the eigenvalues of $U(T,0)$  with modulus equal to 
$r_{0}$ are  semisimple, then estimate \eqref{asintotico} holds 
with $\omega = \omega_{0}(U)$ and consequently 
\eqref{expdecay} holds with $\omega = \omega_{0}(U)$.
If some of such eigenvalues $\mu$ is not 
semisimple, the  eigenvalue $\lambda = \log \mu/T $ of $ G_\#$ is not 
semisimple by Corollary~\ref{Cor:eigen}, and for every $v\in $ Ker $(\lambda I - 
G_\#)^{2} $  such that  $ \lambda v - 
G_\# v = u \in $ Ker $G_\#\setminus \{ 0\}$ we have 
$\cP_{\tau }^{\#}v$ $=$ ${\rm e}^{\lambda  \tau}v$ $- \tau {\rm e}^{\lambda 
\tau}u$ for each $\tau >0$, so that for  
$\omega = \omega_0(U)$ there does not exist any $M>0$ such that
\eqref{expdecay} holds.
\end{proof}

Formula \eqref{expdecay} improves the convergence result of 
\cite[Prop.~6.4]{DPL06}, obtained by different methods.


\subsection{Spectral gap of $G$ and  asymptotic behavior of 
$(\cP_\tau )_{\tau\geq0}$.}

In this section the functions $A$, $B$, $f$ are not necessarily 
periodic but just bounded. 
Although our results are not as precise as in the periodic case, 
still the Poincar\'e type inequality of the next theorem yields
information on the asymptotic behavior of $(\cP_\tau )_{\tau\geq 0}$.

We use the notation of \S 2.3, setting again for each $u\in 
L^{2}(\R^{1+n}, \nu)$
$$(\Pi u)(t,x) = M_{t}u(t, \cdot) = \int_{\R^{n}}u(t,x) \dd \nu_{t}, 
\quad t\in \R, \;x\in \R^{n}.$$
$\Pi$ is still an orthogonal projection, that maps $L^{2}(\R^{1+n}, \nu)$ into its 
subspace of the functions independent of $x$, isomorphic to 
$L^{2}(\R, \ddc t)$.

\begin{theorem}
    \label{Th:Poincare}
    For each $\omega \in (\omega_{0}(U), 0)$ let $M=M(\omega)$ be given by
     \eqref{exp_decay}. Set moreover  $C:= \sup_{  t\in \R}\|B (t)\|$. 
     Then for each $u\in D(G)$ we have
\begin{equation}
\label{Poincare}
\int_{\R^{1+n}}
(u(t,x)-\Pi u(t))^2 \dd \nu \le \frac{M^{2}C^{2}}{2\omega}
\int_{\R^{1+n}} |D_xu(t,x)|^2\dd \nu .
\end{equation}
\end{theorem}

A similar inequality was proved in 
\cite[Thm. 6.3]{DPL06} in the periodic case for functions in $D(G_\#)$, 
but the proof is the same for functions in $D(G)$; one has just to 
replace the core used in \cite{DPL06} by $D(G_{0})$ and the integrals over $(0,T)\times 
\R^{n}$ by integrals over $ \R^{1+n}$. So, we omit the proof. 

Once estimate \eqref{Poincare} is available, a convergence result 
follows in a more or less standard way.

\begin{corollary}
    \label{Cor:growthbound}
Let $\omega$,  $M $,  $C$ be as in Theorem \ref{Th:Poincare}, and let 
$\mu_{0}$ be the constant in \eqref{invertb}. 
For each $u\in L^{2}(\R^{1+n}, \nu)$ we have
\begin{equation}
\label{e7.4}
\|\mathcal{P}_\tau (u- \Pi u )\|_{L^{2}(\R^{1+n}, \nu)}
\leq e^{\omega  \mu_{0}^{2} \tau /M^{2}C^{2} }
\| u - \Pi u \|_{L^{2}(\R^{1+n}, \nu)}, \quad \tau 
>0.
\end{equation}
\end{corollary}
Again, the proof is the same of \cite[Prop. 6.4]{DPL06}, and it is 
omitted. 

Corollary \ref{Cor:growthbound} shows that the growth bound of 
$\mathcal{P}_\tau (I-\Pi)$ does not exceed the number $c_{0}$ defined by
\begin{equation}
\label{c0}
c_{0}= \inf \bigg\{ \frac{\omega \mu_{0}^{2}}{M(\omega)^{2}C^{2}}\,: \;\omega \in 
(\omega_{0}(U), 0)  \bigg\}.
\end{equation}
But $c_{0}$  does not seem to be optimal. By estimates \eqref{stime}
the asymptotic behavior of the space derivatives of $\mathcal{P}_\tau u$ is 
the same of the periodic case, and this suggests that the growth 
bound of $\mathcal{P}_\tau (I-\Pi)$ should be equal to $\omega_{0}(U)$.

Now we can prove some spectral properties of $G$. 

\begin{proposition}
The following statements hold true.
\begin{itemize}
    
 \item[(i)] 	The spectrum of $G$ is invariant under translations along $i\R$.
 
 \item[(ii)] $i\R\subset \sigma(G)$, and $\lambda I-G$ is one to one 
 for each $\lambda \in i\R$. The associated spectral projection is 
 $\Pi$. 
 
 \item[(iii)] $\sigma(G) \cap \{\lambda \in \C:\; \RE\lambda \in 
 ( c_{0},0)\} = \varnothing $.  

 \item[(iv)] If the data $A $, $B$, $f$ are $T$-periodic,
				 then 
 $\frac1T\log\sigma(U(T,0))+ i\R \subset \sigma(G)$.
\end{itemize}
\end{proposition}
\begin{proof}
    For every $\xi\in \R$ let us consider the 
    unitary operator $T_{\xi}$ in $L^2( \R^{1+n},\nu)$ defined by 
    $T_{\xi}u(t,x) = e^{it\xi }u(t,x)$. Since the spectrum of $G$ is equal to the 
    spectrum of $(T_{\xi})^{-1}GT_{\xi} = G +i\xi I$,  
     statement (i) follows.

\vspace{1mm}
Let us split  $L^2( \R^{1+n},\nu)$ in the direct sum 
$$L^2( \R^{1+n},\nu) = (I-\Pi)(L^2( \R^{1+n},\nu)) \oplus \Pi (L^2( 
\R^{1+n},\nu)).$$
The semigroup  $\mathcal{P}_\tau$ maps 
$(I-\Pi)(L^2( \R^{1+n},\nu))$ into itself (the proof is the same of 
the periodic case), and the
growth bound of $\mathcal{P}_\tau(I-\Pi)$ is less or equal to 
$c_{0}$, by corollary \ref{Cor:growthbound}. It follows that the 
spectrum of the part of $G$ in $(I-\Pi)(L^2( \R^{1+n},\nu))$ is 
contained in the halfplane $ \{\lambda \in \C:$ Re$\,\lambda \leq c_{0}\}$. 

The part of $G$ in $\Pi (L^2( \R^{1+n},\nu))$ is just the time derivative, with domain isomorphic to $H^{1}(\R, 
dt)$. Its spectrum is $i\R $, and it has no eigenvalues. 
Statements (ii) and (iii) follow. 

\vspace{1mm}

In the periodic case, let $\mu\in \sigma(U(T,0))$. By Proposition \ref{eigen}, $\lambda := 
\log \mu /T $ is an eigenvalue of $G_\#$. Let $u$ be an 
eigenfunction. 
Fix a function  $\theta\in C^{\infty}(\R)$  such that $\theta(t) \equiv 1$ in 
$(-\infty, 0]$, $\theta \equiv 0$ in $[T, +\infty)$, 
and define  
$\theta _k(t ) = \theta(t-kT)$ for $t\geq 0$, $\theta _{k}(t ) = \theta(-t-kT)$ for 
$t\leq 0$. 

Then the functions  $u_{k}(t,x) := u(t,x)\theta_{k}(t)$ belong to 
$D(G)$ and satisfy $(\lambda I-G)  u_{k} (t,x) $ $=$ $ \theta 
_{k}'(t)u(t,x)$, so that  $\|(\lambda I-G)  u_{k} \|_{L^2( 
\R^{1+n},\nu)}$  is bounded by a constant independent of $k$, while 
$\|u_{k}\|_{L^2( 
\R^{1+n},\nu)^{2}} \geq \int_{-kT}^{kT}\int_{\R^{n}} |u(t,x)|^{2}\dd \nu
= 2k \|u\|_{L^2_\#((0,T)\times\R^n,\nu)}$ goes to $\infty$ as $k\to 
\infty$. This shows that $\lambda I -G$ cannot have a bounded inverse, 
so that  $\lambda \in \sigma (G)$.  
\end{proof}

\begin{remark}
\label{solvability2}
{\em For $h\in L^2( \R^{1+n},\nu)$ consider the equation}
$$Gu =h.$$
{\em It is equivalent to the system}
$$\left\{ \begin{array}{ll}
(i)& G (I-\Pi) u = (I-\Pi)h,
\\
\\
(ii)& G \Pi u = \Pi h.
\end{array}\right.
$$
{\em Equation (i) is uniquely solvable with respect to $(I-\Pi) u$, because $0$ is in the resolvent set of the 
part of $G$ in  $(I-\Pi)(L^2( \R^{1+n},\nu))$. Equation (ii) is 
equivalent to}
$$\frac{\ddc}{\ddc t}\Pi u = \Pi h,$$
{\em and it is solvable iff $\Pi h$ has a primitive $\xi$ in $L^{2}(\R, 
\ddc t)$, in this case the solution is unique. }

{\em So, the range of $G$ consists of the functions $h$ such that 
$\Pi h$ has a primitive $\xi$ in $L^{2}(\R, 
\ddc t)$. Therefore, $G$ is not a Fredholm operator. }
\end{remark}

\begin{remark}
\label{autonomous2}
{\em Arguing as in Remark \ref{autonomous}, we obtain that in the 
autonomous case   $A(t)\equiv A$, $f(t)\equiv 0$, 
$B(t)\equiv B$,   the spectrum 
of $G $ consists of a sequence of vertical lines, and precisely}
$$\sigma(G ) = \bigg\{\lambda \in \C: \; \RE \lambda =  
\sum_{j=1}^rn_j\RE \lambda_j ; \;  n_j\in\N  \cup \{0\} \bigg\}$$
{\em where $\lambda_j$, $j=1, \ldots , r$ are the eigenvalues of $A$. 
Since in this case $\omega_{0}(U)$ is equal to the biggest real 
part of the eigenvalues of $A$, then the spectrum 
of $G$ does not contain elements with real part in 
$(\omega_{0}(U),0)$. So, we have the same spectral gap as in the time 
periodic context. }
\end{remark}

In the previous section we deduced asymptotic behavior results for 
$\cP^{\#}_\tau $ from asymptotic behavior of $P_{s,t}$. Now we 
reverse the procedure, deducing asymptotic behavior of $P_{s,t}$ from 
Corollary \ref{Cor:growthbound}.

\begin{theorem}
    \label{th:P(s,t)asintotico}
Let $c_{0}$ be defined by \eqref{c0}. For each $s<t\in \R$ and  $\varphi \in 
L^{2}(\R^{n}, \nu_{t})$ we have
\begin{equation} 
    \label{P(s,t)asintotico}
\|P_{s,t}(\varphi - M_{t}\varphi)\|_{L^{2}(\R^{n}, \nu_{s})} \leq 
  e^{c_{0}(t-s)}\|\varphi \|_{L^{2}(\R^{n}, \nu_{t})} . 
\end{equation}
\end{theorem}
\begin{proof}
The starting point is the continuity of the function 
$s\mapsto \|P_{s,s+\tau}\varphi \|_{L^{2}(\R^{n}, 
\nu_{s})}^{2}$, for each $\tau > 0$ and for each $\varphi \in 
C^{1}_{b}(\R^{n})$.
Once it is established, we get estimate \eqref{P(s,t)asintotico} for 
$\varphi \in C^{1}_{b}(\R^{n})$, arguing as in the case of evolution semigroups 
in a fixed Banach space. Since $C^{1}_{b}(\R^{n})$ is dense in 
$L^{2}(\R^{n}, \nu_{t})$, estimate \eqref{P(s,t)asintotico} follows 
for each $\varphi \in 
L^{2}(\R^{n}, \nu_{t})$. 

\vspace{2mm}

\noindent {\em Step 1: continuity of} $s\mapsto \|P_{s,s+\tau}\varphi \|_{L^{2}(\R^{n}, 
\nu_{s})}^{2}$.

Fix $s$, $s_{0}\in \R$. Changing variables in an obvious way, we 
write 
\begin{equation}
    \label{eq:cont}
\|P_{s,s+\tau}\varphi \|_{L^{2}(\R^{n}, \nu_{s})}^{2}-
 \|P_{s_{0},s_{0}+\tau}\varphi \|_{L^{2}(\R^{n}, \nu_{s_{0}})}^{2}
 = \int_{\R^{n}} (u(s,x)^{2}-u(s_{0},x)^{2})\NN _{0,I}(\ddc x),
\end{equation}
where
$$u(s,x) := P_{s,s+\tau}\varphi(Q(s, -\infty)^{1/2}x + g(s, -\infty)).$$
Since $\|u\|_{\infty}\leq \|\varphi\|_{\infty}$, then 
$|u(s,x)^{2}-u(s_{0},x)^{2}| \leq 2 \|\varphi\|_{\infty}|u(s,x)-u(s_{0}, 
x)|$. We estimate $|u(s,x)-u(s_{0},x)|$ 
changing again variables, as follows:
$$|u(s,x)-u(s_{0},x)|   \leq  
$$
$$\begin{array}{l}
 \displaystyle{ \int_{\R^{n}}}   \big|  \varphi(Q(s+\tau,s)^{1/2}y + U(s+\tau, 
s)(Q(s,-\infty)^{1/2}x + g(s, -\infty)) + g(s+\tau, s))  
\\
  - \varphi(Q(s_{0}+\tau,s_{0})^{1/2}y + U(s_{0}+\tau, 
s_{0})(Q(s_{0},-\infty)^{1/2}x + g(s_{0}, -\infty)) + g(s_{0}+\tau, 
s_{0})) \big|   \\
\\
 \NN _{0,I}(\ddc y)\end{array}$$

$$\begin{array}{ll}
    \leq   \|\,|D\varphi|\,\|_{\infty}  \bigg( & \displaystyle{\frac{2^{n/2}}{\pi^{n/2}} }
\|Q(s+\tau,s)^{1/2} - Q(s_{0}+\tau,s_{0})^{1/2}\| + 
\\
& \|U(s+\tau,s)Q(s, -\infty)^{1/2} - U(s_{0}+\tau, 
s_{0})Q(s_{0}+\tau,s_{0})^{1/2}\| \,|x|+
\\
& |g(s+\tau, s) - g(s_{0}+\tau, s_{0})|\bigg)
\end{array}$$
Using this estimate, we see that the integral  in \eqref{eq:cont} 
goes to $0$ as $s\to s_{0}$  by dominated convergence.  

\vspace{2mm}

\noindent {\em Step 2: conclusion.}

Fix $t\in \R$ and $\xi \in C^{\infty}_{c}(\R )$ such that $\xi(t) 
=1$. Set 
$$u(s,x):= \xi(s) \varphi(x), \quad s\in \R, \;x\in \R^{n}.$$
Then $u\in L^{2}(\R^{1+n}, \nu)$. We recall that 
$$(\cP_\tau (u-\Pi u))(s,x) = P_{s,s+\tau} u(s+\tau, \cdot)(x) - 
M_{s+\tau}u(s+\tau, \cdot) = \xi(s+\tau)(P_{s,s+\tau}\varphi(x) - 
M_{s+\tau}\varphi),$$
so that 
$$\|\cP_\tau (u-\Pi u)(s, \cdot)\|_{L^{2}(\R^{n}, 
\nu_{s})}^{2} = \xi(s+\tau)^{2}\bigg(\int_{\R^{n}}(P_{s,s+\tau}\varphi(x))^{2}
\nu_{s}(\ddc x) -\bigg( \int_{\R^{n}} \varphi(x) \nu_{s+\tau}(\ddc 
x)\bigg)^{2}\bigg).$$
Therefore, for each $\tau > 0$ the function $s\mapsto \|\cP_\tau (u-\Pi u)(s, \cdot)\|_{L^{2}(\R^{n}, 
\nu_{s})}^{2} $ is continuous. This is true also at $\tau =0$, since 
\begin{align*}
    \| (u-\Pi u)(s, \cdot)\|_{L^{2}(\R^{n}, \nu_{s})}^{2} 
= \| \xi(s) (\varphi - M_{s}\varphi)\|_{L^{2}(\R^{n}, \nu_{s})}^{2}
\\
= |\xi(s)|^2 \bigg( \int_{\R^{n}} \varphi(x)^{2} \nu_{s}(\ddc x) - 
\bigg( \int_{\R^{n}} \varphi(x) \nu_{s}(\ddc x))\bigg)^{2}\bigg).
\end{align*}
Hence, we have
\begin{align*}
\|P_{s,t}(I-M_{t})\varphi\|^2_{L^2(\R^n,\nu_s)}
		 &=\| \cP_{t-s}(u-\Pi u) (s, \cdot 
		 )\|^2_{L^2(\R^n,\nu_s)}\\
		 & =\lim\limits_{\varepsilon\to0^+}\frac1\varepsilon
		 \int\limits_s^{s+\varepsilon}\| \cP_{t-s}(u-\Pi u) (\eta, \cdot) 
		\|^2_{L^2(\R^n,\nu_{\eta})}\dd \eta\\
		 &=\lim\limits_{\varepsilon\to
		 0^+}\frac1\varepsilon\|\chi_{[s,s+\varepsilon]}\cP_{t-s}(u-\Pi u) \|_{L^2(\R^{n+1},\nu)}^{2}
		 \\
		 &=\lim\limits_{\varepsilon\to
		 0^+}\frac1\varepsilon\|\cP _{t-s}(\chi_{[t,t+\varepsilon]}(u-\Pi 
		 u)) \|_{L^2(\R^{n+1},\nu)}^{2}\\
		 &\leq  {\rm e}^{2c_{0}(t-s)}\lim\limits_{\varepsilon\to
		 0^+}\frac1\varepsilon\| \chi_{[t,t+\varepsilon]}(u-\Pi u) \|_{L^2(\R^{n+1},\nu)}^{2}\\
		 &= {\rm e}^{2c_{0}(t-s)}\lim\limits_{\varepsilon\to
		 0^+}\frac1\varepsilon \int_{t}^{t+\varepsilon}\xi(\eta)^{2}
		 \|(\varphi - M_{\eta}\varphi)\|_{L^{2}(\R^{n}, \nu_{\eta})}^{2}
		 \\
		 & = {\rm e}^{2c_{0}(t-s)}\|\varphi - 
		 M_{t}\varphi\|_{L^2(\R^{n},\nu_t)}^{2}
\end{align*} 
and \eqref{P(s,t)asintotico} follows. \end{proof}


\section{Hypercontractivity}

In this section the data $A$, $B$, $f$ are bounded but not necessarily 
periodic. 

Since   $ \cP _\tau $ acts as a translation semigroup  
in the time variable, it cannot improve $\nu$-summability. 
Thus, it seems hard to get hypercontractivity 
estimates for $P_{s,t}$ from properties of  $ \cP _\tau $.
In fact, we follow the ideas of \cite{Gro75}, adapting 
his procedure to the time depending case: fixed any $t\in \R$ and 
$q>1$, we look for a differentiable function $p:(-\infty, t]\mapsto 
[q, +\infty)$ such that $p(t) = q$ and 
$$\frac{\partial}{\partial s}\,\| P_{s,t}\varphi\|_{L^{p(s 
)}(\R^n,\nu_s)}\geq 0, \quad s\leq t$$ 
for all good (e.g., exponential) functions 
$\varphi$. If such a $p$ exists, we get $\| P_{s,t}\varphi\|_{L^{p(s 
)}(\R^n,\nu_s)} \leq \|\varphi\|_{L^{q}(\R^n,\nu_t)} $ for all 
exponential functions, and hence, by density, for all $\varphi\in L^{q}(\R^n,\nu_t)$. 

In the time independent case, hypercontractivity of a semigroup is equivalent to the 
occurrence of a logarithmic Sobolev inequality for its invariant 
measure (\cite{Gro75}). Since our measures $\nu_{t}$ are 
Gaussian, they satisfy logarithmic Sobolev inequalities, which are the 
starting point of the procedure. 
As in the autonomous case, what we need are 
 log-Sobolev inequalities expressed in terms of the quadratic forms associated to 
 the operators $L(t)$. Dealing with the nonautonomous case, an 
 additional term appears in the quadratic form,  i.e. we have 
 \begin{equation}
     \label{quadraticform}
\int\limits_{\R^n}\varphi\, L(t) \varphi \, \nu_t(\ddc
 x)=-\frac 12\int\limits_{\R^n}|B^*(t)\nabla \varphi |^2\nu_t(\ddc
 x)
 -\frac 12 \int\limits_{\R^n}\varphi ^2\,\partial_t\rho(x,t)\dd x, 
 \quad \varphi\in H^{2}(\R^{n}, \nu_{t}),
 \end{equation}
as a consequence of \cite[Lemma 2.4]{GL07}, and this produces an 
additional term in the log-Sobolev inequalities. More precisely, the 
following lemma holds.

\begin{lemma}
\label{lem:logsob2}
For $p\in(1,\infty)$, $t\in\R$
and $\varphi\in W^{2,p}(\R^n,\nu_t)$, we have
\begin{equation}
    \label{logSob}
\begin{array}{l}
\displaystyle{\int\limits_{\R^n}|\varphi(x)|^p\log(|\varphi(x)|)\nu_t(\ddc x)
\leq \|\varphi\|_{L^p(\R^n,\nu_t)}^p\log(\|\varphi\|_{L^p(\R^n,\nu_t)})}
\\
\\
\displaystyle{ + c(p,t)\bigg(\RE\langle
- L(t)\varphi,\varphi_p\rangle_{L^2(\R^n,\nu_t)} + \frac1p\int\limits_{\R^n}
|\varphi(x)|^p\partial_t\rho\dd
x\bigg).}
\end{array}
\end{equation}
Here, $\varphi_p = |\varphi|^{p-2}\varphi$ and 
\begin{equation}
    \label{c(p,t)}
c(p,t) =  \frac {p }{ p-1 } \|Q^{1/2}(t,-\infty)B^{*-1}(t )\|^{2}  . 
\end{equation}
\end{lemma}
\begin{proof}
The starting point is the logarithmic Sobolev inequality 
$$\int\limits_{\R^n}|\psi(x)|^2\log(|\psi(x)|)\nu_t(\ddc x)
\leq   \|Q^{1/2}(t,-\infty)\nabla \psi\|^2_{L^2(\R^n,\nu_t)} 
+\|\psi\|^{2}_{L^2(\R^n,\nu_t)}\log\|\psi\|_{L^2(\R^n,\nu_t)},$$
valid for any $t\in \R$ and $\psi\in H^{1 }(\R^n,\nu_t)$, 
that follows from the well known logarithmic Sobolev inequality for 
the Gaussian measure $\NN (0,I)$ (e.g., \cite[formula 
(1.2)]{Gro75}) via the standard change of 
variables already used in the proof of Theorem \ref{th:P(s,t)asintotico}.  
Since $B^{*}(t)$ is 
invertible, we get 
\begin{equation}
\label{eq:logsobNostra}
\begin{array}{lll} 
\displaystyle{\int\limits_{\R^n}|\psi(x)|^2\log(|\psi(x)|)\nu_t(\ddc x)}
& \leq & 
\displaystyle{ \|Q^{1/2}(t,-\infty)B^{*-1}(t )\|^{2}
\int\limits_{\R^n}|B^{*}(t)\nabla 
\psi(x) |^2\nu_t(\ddc x) }
\\
\\
& & +\|\psi\|^2_{L^2(\R^n,\nu_t)}\log\|\psi\|_{L^2(\R^n,\nu_t)}.
\end{array}
\end{equation}
The statement will be obtained applying \eqref{eq:logsobNostra} to 
the functions $\varphi_\varepsilon:=
 (|\varphi|^2+\varepsilon)^{\frac p4}$, and then letting 
$\varepsilon\to 0^{+}$. To this aim, we have to estimate the integrals
$\int_{\R^n}|B^{*}(t)\nabla 
\varphi_\varepsilon  |^2\nu_t(\ddc x) $.	
Here and in the following, we suppress the dependency of $\varphi$
and $\varphi_\varepsilon$ on $x$. An easy calculation shows that 
\begin{align}
\label{fe1}
\partial_j \varphi_\varepsilon 
& =\frac p4 (|\varphi|^2+\varepsilon)^{\frac p4-1}\partial_j |\varphi|^2
\\
\label{fe2}
(B^*(t)\nabla \varphi_\varepsilon)^2 
&= \frac{p^2}{16}(|\varphi|^2+\varepsilon)^{\frac p2-2}
\left(B^*(t)\nabla |\varphi|^2 \right)^2 
\\
\label{fe3}
\partial_{ij}\varphi_\varepsilon
&=\frac p4 \left(\frac p4-1\right)(|\varphi|^2+\varepsilon)^{\frac p4 -2}
\partial_i |\varphi|^2\cdot \partial_j |\varphi|^2
+\frac p4 ( |\varphi|^2+\varepsilon)^{\frac p4-1}\partial_{ij}|\varphi|^2.
\end{align}
It follows from \eqref{fe3} and  from the identity
$L(t)(\varphi\overline{\varphi})=2\,\Re \overline{\varphi}\,L(t)\varphi+
|B(t)^*\nabla \varphi|^2$
that
\begin{align*}
L(t)\varphi_\varepsilon
=&\frac p8 \left(\frac p4-1\right)(|\varphi|^2+\varepsilon)^{\frac
p4-2}\left(B^*(t)\nabla |\varphi|^2\right)^2
+\frac p4(|\varphi|^2+\varepsilon)^{\frac p4-1} L(t)|\varphi|^2
\\
=&\frac p8 \left(\frac p4-1\right)(|\varphi|^2+\varepsilon)^{\frac
p4-2}\left(B^*(t)\nabla |\varphi|^2\right)^2
+\frac  p2\RE(|\varphi|^2+\varepsilon)^{\frac p4-1} \overline
\varphi L(t)\varphi
\\
&+\frac  p4(|\varphi|^2+\varepsilon)^{\frac p4-1}|B(t)^*\nabla
\varphi|^2.
\end{align*}
Since $|\varphi|^2|B^*(t)\nabla \varphi|^2\geq 
\frac 14(B^*(t)\nabla|\varphi|^2)^2$, we obtain
\begin{align*}
L(t)\varphi_\varepsilon
=&\frac p8 \left(\frac p4-1\right) (|\varphi|^2+\varepsilon)^{\frac
p4-2}\left(B^*(t)\nabla |\varphi|^2\right)^2
+\frac  p2\RE(|\varphi|^2+\varepsilon)^{\frac p4-1} 
\overline \varphi L(t)\varphi
\\
&+\frac  p4(|\varphi|^2+\varepsilon)^{\frac p4-2}|\varphi|^2|B(t)^*\nabla
\varphi|^2 +\frac  p4(|\varphi|^2+\varepsilon)^{\frac p4-2}\varepsilon
|B(t)^*\nabla \varphi|^2
\\
\geq &\frac {p^2-2p}{32}(|\varphi|^2+\varepsilon)^{\frac
p4-2}\left(B^*(t)\nabla |\varphi|^2\right)^2
+\frac  p2\RE(|\varphi|^2+\varepsilon)^{\frac p4-1} \overline
\varphi L(t)\varphi.
\end{align*}
Finally, \eqref{fe2} yields 
\begin{align*}
L(t)\varphi_\varepsilon \geq
& \; \frac{p-2}{ 2p}\varphi_\varepsilon^{-1}
\left(B^*(t)\nabla \varphi_\varepsilon\right)^2
+\frac  p2\,\RE(|\varphi|^2+\varepsilon)^{\frac p4-1} \overline
\varphi L(t)\varphi.
\end{align*}
Applying  the identity \eqref{quadraticform} to $\varphi_\varepsilon$ 
we obtain
\begin{align*}
\int\limits_{\R^n}|B^*\nabla \varphi_\varepsilon|^2\nu_t(\ddc
x)
\leq&\int\limits_{\R^n}\varphi_\varepsilon^2 \, \partial_t\rho(x,t)\dd x
-\frac{p-2}{p}\int\limits_{\R^n}
(B^*(t)\nabla \varphi_\varepsilon)^2\nu_t(\ddc x)
\\
&-p\, \RE\int\limits_{\R^n}  (|\varphi|^2+\varepsilon)^{\frac p2-1} \overline
\varphi L(t)\varphi \,\nu_t(\ddc x).
\end{align*}
This implies
\begin{align*}
& \int\limits_{\R^n}|B^*(t)\nabla \varphi_\varepsilon|^2\nu_t(\ddc x)
\\
& \leq -\frac {p^2}{2(p-1)} 
\bigg(   \RE
\int\limits_{\R^n}(\varphi^2+\varepsilon)^{\frac     
p2-1}\overline{\varphi}\,
L(t)\varphi\,\nu_t(\ddc x)
-\frac 1p \int\limits_{\R^n}\varphi_\varepsilon^2\,\partial_t\rho(x,t)\dd 
x\bigg).
\end{align*}
Replacing this estimate in \eqref{eq:logsobNostra} and 
letting $\varepsilon$ tend to $0$, the lemma follows.
\end{proof}

Next, we prove a variant of \cite[Lemma 1.1]{Gro75}. Again, we have to
deal with an additional term.

\begin{lemma}
\label{lem:logsob3}
Let $t\in\R$, $a\in (0, +\infty]$ and
$I=(t-a,t]$. Assume that 
$p\in C^1( I )$ with $p(s)>1$ for $s\in I$, $u(\cdot,x)\in
C^1( I )$ for all $x\in\R^n$ and $u(s,\cdot)\not\equiv 0$ for $s\in I$.
Moreover, assume that  there are $C,k>0$ such that 
\begin{align*}
\max\left\{
|u(s,x)|,|\partial_su(s,x)|\right)\leq C|x|^k,\quad s\in I, \;
\quad x\in\R.
\end{align*}
Then the function $\alpha:I\to \R$ defined by  $\alpha(s)=\|u(s,\cdot)\|_{L^{p(s)}(\R^n,\nu_s)}$ is differentiable in
$I$ and
\begin{align*}
 \alpha'(s) = & \, \alpha(s)^{1-p(s)}
\bigg\{ 
\RE\langle
\partial_s
u(s,\cdot),u_{p(s)}(s,\cdot)\rangle_{L^2(\R^n,\nu_s)}+\frac
1{p(s)}
\int\limits_{\R^n}|u(s,x)|^{p(s)}\partial_s\rho\dd x
\\
&+\frac{p'(s)}{p(s)} \bigg(
    \int\limits_{\R^n}|u(s,x)|^{p(s)}\log(|u(s,x)|)\nu_s(\ddc x)
- \alpha(s)^{p(s)}\log(\alpha(s)) \bigg) \bigg\}.
\end{align*}
\end{lemma}
\begin{proof}
We calculate
$$\begin{array}{l}
\displaystyle{\frac\partial{\partial s} \left( |u(s,x)|^{p(s)}\rho(s,x) \right)}
\\
\\
= \displaystyle{\left(p'(s)\log(|u(s,x)|)|u( s,x)|^{p(s)}
+p(s)\frac\partial{\partial s} u (s,x)|u(s,x)|^{p(s)-2}
u(s,x)\right)  \rho(s,x)  }
\\
\\
\displaystyle{+|u(s,x)|^{p(s)}\frac \partial {\partial s}\rho(s,x),}
\quad s\in I.
\end{array}$$
By assumption, there exists $h\in L^1(\R^{n})$ such that
\begin{align*}
\max\bigg\{ |u(s,x)|^{p(s)}\rho(s,x),
\frac\partial{\partial s}\left(
|u(s,x)|^{p(s)}\rho(s,x)\right)
\bigg\}\leq h(x),\quad s\in I, \; x\in\R^n.
\end{align*}
Hence, the assertion follows from Lebesgue's dominated
convergence theorem and the chain rule.
\end{proof}

Now we are able to prove the hypercontractivity of
$(P_{s,t})_{s\leq t}$.

\begin{theorem}
    \label{Th:hyper}
Let $q\in( 1,\infty )$, $t\in \R$  and let 
$p(s,t)$ be the solution of 
$$p'(s) = -\frac{p(s) }{c(p,s)}, \; s\leq 
t; \quad p(t)=q.$$
Then for $s<t$, 
 $P_{s,t}$ maps $L^q(\R^n,\nu_{t})$ into $L^{p(s,t)}(\R^n,\nu_s)$ and
\begin{align*}
\| P_{s,t}\varphi\|_{L^{p(s,t)}(\R^n,\nu_s)}\leq
\|\varphi\|_{L^q(\R^n,\nu_{t})}, \quad \varphi\in
L^q(\R^n,\nu_{t}) .
\end{align*}
\end{theorem}
\begin{proof}
Fix $t\in\R$ and
let $\varphi\in {\mathrm span}\;\{ {\rm e}^{i\langle
k,x\rangle}:k\in \R^n\}$. Set $p(s)=p(s,t)$,
$u(s,\cdot)=P_{s,t}\varphi$ and
$\alpha(s)=\|P_{s,t}\varphi\|_{L^{p(s)}(\R^n,\nu_s)}$. Since 
$$P_{s,t}\varphi_k(x)={\rm e}^{i\langle
g(t,s)+U(t,s)x,k\rangle-\frac12\langle
Q(t,s)k,k\rangle},$$
for $\varphi_k(x) ={\rm e}^{i\langle
k,x\rangle}$, then the functions $\alpha$ and  $p$ satisfy the assumptions 
of Lemma~\ref{lem:logsob3}. Using Lemma~\ref{lem:logsob3}, 
we get 
\begin{align*}
& \alpha'(s) 
=  \alpha(s)^{1-p(s)}
\bigg\{
\RE\langle
-L(s)u(s,\cdot),u_{p(s)}(s,\cdot )\rangle_{L^2(\R^n,\nu_s)}
+\frac 1{p(s)}
\int\limits_{\R^n}|u(s,x)|^{p(s)}\partial_s\rho(s,x)\dd x 
\\
& 
+\frac{p'(s) }{ p(s) } 
\bigg(\int\limits_{\R^n}|u(s,\cdot)|^{p(s)}\log(|u(s,\cdot)|)\nu_t(\ddc
x) -\|u(s,\cdot)\|_{L^{p(s)}(\R^n,\nu_s)}^{p(s)}\log(\|u(s,\cdot)\|_{L^{p(s)}(\R^n,\nu_s)})\bigg)
\bigg\}.
\end{align*}
The choice $p'(s) = -\frac{p(s) }{c(p,s)}$ and inequality \eqref{logSob} 
thus yield  $\frac{d \alpha(s)}{ds}\geq
0$, which implies
\begin{align*}
\|P_{s,t}\varphi\|_{L^{p(s,t)}(\R^n,\nu_s)}=\alpha(s)
\leq
\alpha(t)=\|\varphi\|_{L^{q}(\R^n,\nu_{t})},\quad s\leq t.
\end{align*}
Since ${\mathrm span}\;\{ {\rm e}^{i\langle
k,x\rangle}:k\in \R^n\}$ is dense in $L^q(\R^n,\nu_{t})$, the
proof is complete.
\end{proof}

 \begin{remark}
{\em
The solution 
$p(s,t)$ of  
$$p'(s) = -\frac{p(s) }{c(p,s)}, \; s\leq 
t; \quad p(t)=q$$
is given by}
\begin{align*}
p(s,t)=1+(q-1)\exp\left(
 \int_s^{t} \|Q^{\frac12}(r,-\infty)B^{*-1}(r) \| ^{-2}\dd r\right),\quad
s<t.
\end{align*}
{\em Since $\|Q^{1/2}(r, -\infty)B^{*-1}(r)\|^{2} \leq 
\int_{-\infty}^{r} \|B^{*}(\sigma)U^{*}(\sigma, r) B^{*-1}(r)\|^{2} 
d\sigma$, then for each $\omega \in (\omega_{0}, 0)$ we have}
$$\|Q^{\frac12}(r,-\infty)B^{*-1}(r)\|^2 \leq  
\frac{C^{2}(M(\omega))^{2}}{2\mu_{0}^{2}|\omega|}$$
{\em  with $C=\sup_{t\in \R}\|B(t)\|$. Hence, }
$$p(s,t) \geq 1 + (q-1)e^{2c_{0}(s-t)}, \quad s\leq t,$$
{\em  where $c_{0}$ is the constant defined in  \eqref{c0}.}
\end{remark}

\subsection{Acknowledgements.}
We thank Marco Fuhrman for useful conversations about 
hypercontractivity.

\providecommand{\bysame}{\leavevmode\hbox to3em{\hrulefill}\thinspace}
\providecommand{\MR}{\relax\ifhmode\unskip\space\fi MR }
\providecommand{\MRhref}[2]{%
  \href{http://www.ams.org/mathscinet-getitem?mr=#1}{#2}
}
\providecommand{\href}[2]{#2}

\begin{thebibliography}{MPRS02}


\bibitem[CL99]{CL99}
C.~Chicone and Y.~Latushkin, \emph{Evolution semigroups in dynamical systems
  and differential equations}, Mathematical Surveys and Monographs, vol.~70,
  American Mathematical Society, Providence, RI, 1999.

  
\bibitem[CMG96]{CMG96}
A.~Chojnowska-Michalik and 
B.~Goldys, \emph{Nonsymmetric Ornstein-Uhlenbeck operator a  second 
quantized operator}, J. Math. Kyoto Univ. \textbf{36} (1996), 481--498.
  
  
\bibitem[CP01]{CP99}
P.~Cl{\'e}ment and J.~Pr{\"u}ss, \emph{An operator-valued transference
  principle and maximal regularity on vector-valued {$L\sb p$}-spaces},
  Evolution equations and their applications in physical and life sciences (Bad
  Herrenalb, 1998), Lecture Notes in Pure and Appl. Math., vol. 215, Dekker,
  New York, 2001, pp.~67--87.

\bibitem[DPL06]{DPL06}
G.~Da~Prato and A.~Lunardi, \emph{Ornstein-Uhlenbeck operators with time
  periodic coefficients}, J. Evol. Equ. (to appear). 

\bibitem[DPR05]{DPR05}
  G.~Da~Prato and  M.~R\"ockner, \emph{A note on 
non autonomous stochastic differential equations}, Proceedings of the 
 5th Seminar on Stochastic Analysis, Random Fields and Applications, 
 Ascona 2005, R. Dalang, M. Dozzi, F. Russo Eds. Progress in 
 Probability, Birkh\"auser (to appear). 
 
 \bibitem[Dyn89]{Dyn89}
 E.~B.~Dynkin,    \emph{Three Classes of Infinite-Dimensional 
 Diffusions}, J. Funct. Anal. \textbf{86} (1989), 75--110.
  
\bibitem[EN00]{EN00}
K.-J.~Engel and R.~Nagel, \emph{One-parameter semigroups for linear evolution
  equations}, Graduate Texts in Mathematics, vol. 194, Springer-Verlag, New
  York, 2000.

\bibitem[Fuh98]{Fuh98}
M.~Fuhrman, \emph{Hypercontractivity properties of nonsymmetric 
Ornstein-Uhlenbeck semigroups in Hilbert spaces}, Stoch. Anal. Appl. 
\textbf{16} (1998), 241--260.

\bibitem[GL07]{GL07}
  M.~Geissert and A.~Lunardi, \emph{Invariant Measures 
  and Maximal $L^2$ Regularity for
  Nonautonomous Ornstein-Uhlenbeck Equations}, submitted. 
  
\bibitem[GS72]{GS72}
I.I.~Gikhman and A.V.~Skorohod, \emph{Stochastic differential equations}, 
 Springer-Verlag, 1972.

\bibitem[Gro75]{Gro75}
L.~Gross, \emph{Logarithmic {S}obolev inequalities}, Amer. J.
Math. \textbf{97}  (1975), no.~4, 1061--1083.
  
\bibitem[Hen81]{Hen81}
D.~Henry, \emph{Geometric theory of semilinear parabolic 
equations}, Lect. Notes in Math., vol.  840, Springer-Verlag, New York, 1981.

\bibitem[KS91]{KS91} 
I.~Karatzas and S.~E.~Shreve, \emph{Brownian motion and stochastic 
calculus}, Second edition. Graduate Texts in Mathematics, vol. 113, 
Springer-Verlag, New York, 1991.


\bibitem[LMS95]{LMS95}
Y.~Latushkin and S.~Montgomery-Smith, \emph{Evolutionary semigroups and
  {L}yapunov theorems in {B}anach spaces}, J. Funct. Anal. \textbf{127} (1995),
  (1), 173--197.
 
  
\bibitem[MPP02]{MPP02}
G.~Metafune, D.~Pallara, and E.~Priola, \emph{Spectrum of
  {O}rnstein-{U}hlenbeck operators in {$L\sp p$} spaces with respect to
  invariant measures}, J. Funct. Anal. \textbf{196} (2002), no. 1, 40--60.

\bibitem[MPRS02]{MPRS02}
G.~Metafune, J.~Pr{\"u}ss, A.~Rhandi, and R.~Schnaubelt, \emph{The domain of
  the {O}rnstein-{U}hlenbeck operator on an {$L\sp p$}-space with invariant
  measure}, Ann. Sc. Norm. Super. Pisa Cl. Sci. (5) \textbf{1} (2002), 
  no. 2,
  471--485.

\bibitem[Tri78]{Tri78}
H.~Triebel, \emph{{I}nterpolation {T}heory, {F}unction {S}paces, {D}ifferential
  {O}perators}, North-Holland, Amsterdam, 1978.

\end{thebibliography}

\end{document}